\documentclass[twoside,12pt, letterpaper,reqno]{amsart}%
\usepackage[dvips]{graphics}
\usepackage{times}
\usepackage{mathrsfs}
\usepackage[T1]{fontenc}
\usepackage{latexsym}
\usepackage{epsfig}
\usepackage{hyperref}
\hypersetup{
    pdfborder = {0 0 0}
}
\usepackage{amsmath,amsfonts,amsthm,amssymb,amscd}
\usepackage{pstricks}
\usepackage[myheadings]{fullpage}
\newtheorem{theorem}{Theorem}[section]

\newtheorem{corollary}[theorem]{Corollary}
\newtheorem{definition}[theorem]{Definition}

\newtheorem{lemma}[theorem]{Lemma}

\newtheorem{proposition}[theorem]{Proposition}

\newcommand{\kommentar}[1]{}

\newcommand{\supp}{{\rm supp}}

\newcommand{\hg}{\widehat{g}}

\newcommand{\hphi}{\widehat{\phi}}  %phi^

\newcommand\be{\begin{equation}}
\newcommand\ee{\end{equation}}
\newcommand\bea{\begin{eqnarray}}
\newcommand\eea{\end{eqnarray}}
\newcommand\bi{\begin{itemize}}
\newcommand\ei{\end{itemize}}
\newcommand\ben{\begin{enumerate}}
\newcommand\een{\end{enumerate}}
\newcommand\bc{\begin{center}}
\newcommand\ec{\end{center}}
\newcommand\ba{\begin{array}}
\newcommand\ea{\end{array}}
\def\notdiv{\ \mathbin{\mkern-8mu|\!\!\!\smallsetminus}}
\newcommand{\R}{\ensuremath{\mathbb{R}}}

\newcommand{\fof}{\frac{1}{4}}  %oneforth
\newcommand{\foh}{\frac{1}{2}}  %onehalf
\newtheorem{thm}{Theorem}[section]

\theoremstyle{definition}

\newcommand{\twocase}[5]{#1 \begin{cases} #2 & \text{{\rm #3}}\\ #4
&\text{{\rm #5}} \end{cases}   }

\newcommand{\gep}{\epsilon}

\newcommand{\intii}{\int_{-\infty}^\infty}
\newcommand{\ci}{\frac1{2\pi i}}
\numberwithin{equation}{section}
\begin{document}

%%%%%%%%%%%%%%%%%%%%%%%%%%%%%%%%%%%%%%%%%%%%%%%%%%%%%
%%%%% Begin setting up title, authors, abstract
%%%%%%%%%%%%%%%%%%%%%%%%%%%%%%%%%%%%%%%%%%%%%%%%%%%%%

\begin{center}
{\Large{\textsc{{Modeling Convolutions of $L$-functions}}}}

\medskip

by\\\textsc{Ralph Morrison}

\smallskip

\textsc{Steven J. Miller, Advisor}

\medskip

A thesis submitted in partial fulfillment\\ of the requirements for the 
\\Degree of Bachelor of Arts with Honors\\ in Mathematics

\smallskip

\textsc{Williams College}
\\
Williamstown, Massachusetts

\smallskip

Submitted To Williams College: May 14, 2010

Last Updated: October 31, 2010
\thispagestyle{empty}

\end{center}

%\markboth{authors to list on top of page}{title on top of page}

%%%%%%%%%%%%%%%%%%%%%%%%%%%%%%%%%%%%%%%%%%%%%%%%%%%%%
%%%%% End set up title, authors, abstract
%%%%%%%%%%%%%%%%%%%%%%%%%%%%%%%%%%%%%%%%%%%%%%%%%%%%%

\pagebreak

\pagestyle{plain}

\centerline{
\textsc{{Abstract}}
}

{A number of mathematical methods have been shown to model the zeroes of $L$-functions with remarkable success, including the Ratios Conjecture and Random Matrix Theory. In order to understand the structure of convolutions of families of $L$-functions, we investigate how well these methods model the zeros of such functions.  Our primary focus is the convolution of the $L$-function associated to Ramanujan's tau function with the family of quadratic Dirichlet $L$-functions, for which J.B. Conrey and N.C. Snaith computed the Ratios Conjecture's prediction. Our main result is performing the number theory calculations and verifying these predictions for the one-level density for suitably restricted test functions up to square-root error term. 
Unlike Random Matrix Theory, which only predicts the main term, the Ratios Conjecture detects the arithmetic of the family and makes detailed predictions about their dependence in the lower order terms. Interestingly, while Random Matrix Theory is frequently used to model behavior of L-functions (or at least the main terms), there has been little if any work on the analogue of convolving families of L-functions by convolving random matrix ensembles. We explore one possibility by considering Kronecker products; unfortunately, it appears that this is not the correct random matrix analogue to convolving families..}

\pagebreak

\centerline{\textsc{Acknowledgements}}

First and foremost, I would like to thank Professor Steven J. Miller for being an outstanding advisor.  Without his guidance and support I would never have been able to face down the page-long equations and daunting theoretical concepts that have arisen over the past year.  I would like to thank my second reader Professor Mihai Stoiciu for providing feedback on my work and helping me guide it to its final form, and Professor Carston Botts for his notes on the Accept-Reject method.  I would also like to thank  Eduardo Due\~{n}ez, Duc Khiem Huynh, and Nina Snaith for their advice and assistance via email over the past two semesters.  Finally, I would like to thank my fellow thesis students and the rest of the Williams mathematics department (faculty, students and all) for all their support and for creating and maintaining a fun and intellectually stimulating environment in which to do research.

\pagebreak
\tableofcontents
\setcounter{tocdepth}{5}
\pagebreak
%%%%%%%%%%%%%%%%%%%%%%%%%%%%%%%%%%%%%%%%%%%%%%%
%%%%%%%%%%%%%%%%%%INTRODUCTION%%%%%%%%%%%%%%%%%%%
%%%%%%%%%%%%%%%%%%%%%%%%%%%%%%%%%%%%%%%%%%%%%%%

\section{Introduction}

One of the most important areas in modern number theory is the study of the distribution of the zeroes of $L$-functions, meromorphic functions on the complex plane that are continuations of infinite series. The simplest is the most well-known $L$-function, the Riemann-zeta function. It is defined by
 \be \zeta(s)\ := \ \sum_{n=1}^\infty \frac{1}{n^s}\ = \ \prod_{\text{$p$ prime}}\left(1-\frac{1}{p^s}\right)^{-1} \ee
 for $Re(s)>1$ and extended to a meromorphic function. The extension satisfies a functional equation relating its value at $s$ to its value at $1-s$, and trivially vanishes at the negative even integers (which are called the trivial zeros): 
 \be \xi(s) \ :=\ \foh s(s-1)
\Gamma\left(\frac{s}{2}\right)\pi^{-\frac{s}{2}} \zeta(s) \ = \ \xi(1-s).\ee
 The Riemann Hypothesis, often considered the most important open question in mathematics, is the conjecture that all non-trivial zeros of $\xi(s)$ have real part equal to $1/2$. The distribution of the zeros of this and other $L$-functions encode crucial number theoretic information on subjects ranging from the distribution of the primes to properties of class numbers and even mirror the energy levels of neutrons in quantum mechanics, suggesting a deep connection between this branch of mathematics and nuclear physics.  As proofs of properties of these zeros are often out of reach of rigorous methods, methods of modeling these zeros are vital in understanding and formulating appropriate conjectures about $L$-functions. A familiarity with the standard properties of $L$-functions is important in understanding the content and results of this thesis, though intuitive interpretations will be offered whenever appropraite. (See \cite{IK,MT-B} for background on $L$-functions and \cite{FM,Ha} for  the history of the interplay between nuclear physics and zeros of the Riemann zeta function.)

The particular object we will study is the one-level density of the low lying zeros of a family of $L$-functions, which relates sums of an even Schwartz function $\phi$ at the zeros of the $L$-function to sums of the Fourier transform $\hphi$ at the primes. As $\phi$ is a Schwartz function, it vanishes rapidly as $|x|\rightarrow\infty$.  Intuitively, this will be the window through which we observe the low-lying zeros.  Ideally, we would like to use a delta spike instead of a Schwartz test function to get a perfect picture at a point;  however, the delta spike has a Fourier transform of infinite support, which makes such a function inapplicable as the resulting sums of the Fourier transform cannot be evaluated. Following \cite{ILS}, we study the one-level density for an $L$-function $f$, defined by
\be D(f,\phi)\ := \ \sum_{\gamma_f}\phi\left(\frac{\gamma_f L}{\pi}\right);
\ee here $1/2 + i\gamma_f$ runs over the non-trivial zeros of the $L$-function (which under the Generalized Riemann Hypothesis all have $\gamma \in \R$) and $\frac{L}{\pi}$ is a scaling factor (defined explicitly in equation \ref{definitionl}) that measures the spacings between zeros near the central point. As each $L$-function only has a bounded number of zeros within this distance of the central point, it is necessary to average the one-level density over all $f$ in a family $\mathcal{F}$. This allows us to use results from number theory to determine the behavior on average near the central point $1/2$. The exact nature of just what constitutes a family is still being determined, but standard examples include $L$-functions attached to Dirichlet characters, cuspidal newforms, and families of elliptic curves.

We assume our family of $L$-functions $\mathcal{F}$ can be ordered by conductor, and denote by $\mathcal{F}(Q)$ all elements of the family whose conductor is at most $Q$. The quantity of interest ends up being the limit of
\be \frac{1}{|\mathcal{F}(Q)|}\sum_{f\in\mathcal{F}(Q)}\sum_{\gamma}\phi\left(\frac{\gamma L}{\pi}\right)
\ee
as $Q\rightarrow\infty$. Thus we consider the limiting behavior of the average of the one-level densities as the conductors grow.

For a ``nice'' family of $L$-functions, Random Matrix Theory (see \cite{KaSa1,KaSa2}) predicts that the behavior of the zeros as the conductors tend to infinity agree with the $N\to\infty$ scaling limits of a classical compact group of $N\times N$ matrices, most often either unitary, symplectic, or a type of orthogonal (even, odd or mixed). Given two families of $L$-functions $\mathcal{F}$ and $\mathcal{G}$, the Rankin-Selberg convolution $\mathcal{F}\times\mathcal{G}$ is a new family of $L$-functions built from elements of $\mathcal{F}$ and $\mathcal{G}$.  This is a natural type of $L$-function family to study, and is likely to be accessible in the simplest non-trivial case of convolving a family of size $1$ with another family.  An interesting feature of these convolutions was found by Miller and Due\~nez in \cite{DM2}, namely that for ``nice'' families of $L$-functions $\mathcal{F}$ and $\mathcal{G}$, the underlying symmetry groups of $\mathcal{F}$ and $\mathcal{G}$ determine the underlying symmetry group of $\mathcal{F}\times\mathcal{G}$ in a simple, multiplicative way. Specifically, to each family $\mathcal{F}$ is associated a symmetry constant $c_{\mathcal{F}}$ (0 for unitary, 1 for symplectic and -1 for orthogonal) and $c_{\mathcal{F}\times\mathcal{G}} = c_{\mathcal{F}} \cdot c_{\mathcal{G}}$. Unfortunately this only leads to predictions for the main term of the one-level density, and it is in the lower order terms that the arithmetic of the families surface.

In this thesis we focus on testing the Ratios Conjecture's power of modeling the  convolution a family of size $1$ with another family (Sections \S\ref{section2} and \S\ref{section3}). This will allow us to see how the arithmetic of our family enters. Additionally, as Random Matrix Theory has successfully predicted numerous properties of $L$-functions, we try and find the random matrix analogue of convolving two families. To our knowledge this has yet to be investigated in the literature. In Section \S\ref{section4} we report on numerical investigations of the Kronecker products of families of random matrices, which is a natural candidate to model convolutions.

\subsection{The Ratios Conjecture}
The $L$-function Ratios Conjecture of Conrey, Farmer and Zirnbauer \cite{CFZ1,CFZ2} (see also \cite{CS1} for many worked out examples of the conjecture's prediction) are formulas for the averages over families of $L$-functions of ratios of products of shifted $L$-functions.  Their ``recipe'' for performing these calculations starts by using the approximate functional equation, where the error term is discarded, to expand the $L$-functions in the numerator; the $L$-functions in the denominator are expanded via the Mobius function. They then average over the family, and retain only the diagonal pieces. These are restricted sums over integers, but are then completed and extended to sums over all integers; again the error term introduced is ignored.  These methods, far simpler to implement than rigorous analysis, have easily predicted the answers to many difficult computations, and have shown remarkable accuracy. The resulting formulas make detailed predictions on numerous problems, ranging from moments to spacings between adjacent zeros and values of $L$-functions.

A standard test of the Ratios Conjecture is to compare the Ratios Conjecture's predictions for the one-level density  of a family of $L$-functions with rigorous calculation. Agreement has been found (for suitably restricted test functions) for families of Dirichlet $L$-functions and cuspidal newforms (see \cite{GJMMNPP,Mil3,Mil5,MilMo}).  In addition to strengthening the credibility of the conjecture, these calculations provide insight into the significance of the terms that arise in the number theoretic calculations whose corresponding terms in the Ratios Conjecture's predictions are more clearly understandable. For example, in \cite{Mil3} the Ratios Conjecture's prediction allows interpretation of a lower order term in the behavior of the family of quadratic Dirichlet characters as arising from the non-trivial zeros of the Riemann zeta function.

Our primary object of study is the collection of quadratic twists of the $L$-function associated to Ramanujan's tau function,  a family that can be viewed as the convolution of the family of quadratic Dirichlet $L$-functions with the family consisting solely of the tau $L$-function.  The Ratios Conjecture's prediction for this family was computed by Conrey and Snaith in \cite{CS1}.  We perform the number theoretic calculations of the zero statistics for the one-level density for this family, and compare our results to the Ratios Conjecture's prediction. Our main result is the following.

\begin{thm}\label{thm:mainnumbequalsratios} Consider the family of quadratic twists of the tau $L$-function with even fundamental discriminants $d \le X$; denote the number of such $d$ by $X^\ast$ (which is essential a constant times $X$). For $\supp(\hphi) \subset (-\sigma, \sigma)$ with $\sigma < 1$, the one-level density equals
\begin{align}
&\frac{1}{X^*}\sum_{d\leq X}\sum_{\gamma_d}g\left(\gamma_d\frac{L}{\pi}\right)\nonumber
\\ \ = \ &\frac{1}{2LX^* }\int_{-\infty}^\infty g(\nu)\left(\sum_{d\leq {X}}\left[2\log\left(\frac{d}{2\pi}\right)+\frac{\Gamma'}{\Gamma}\left(6+i\frac{\pi\nu}{L}\right)+\frac{\Gamma'}{\Gamma}\left(6-i\frac{\pi\nu}{L}\right)\right]\right.\nonumber
\\&+2\left(-\sum_{p}\sum_{k=1}^\infty  \frac{(\alpha_p^{2k}+\overline{\alpha}_p^{2k})) \log p}{p^{k(1+\frac{2\pi i\nu}{L})} }\right.
\left.\left.+ \sum_{p} \frac{\log p}{(p + 1)} \sum_{k = 1}^\infty \frac{(\alpha_p^{2k} + \overline{\alpha}_p^{2k}) }{p^{k(1 + \frac{2 \pi i \nu}{L})}} \right)\right)d\nu \nonumber
\\&+ O(X^{-(1-\sigma)/2}\log^6 X)\label{mainresult}
\end{align}
which agrees with the Ratios Conjecture's prediction up to an error term of size $O(X^{-(1-\sigma)/2+\varepsilon})$ for any $\varepsilon>0$ (essentially the error term of the expression).
\end{thm}

In addition to being of interest in its own right, understanding this family is useful for investigations of elliptic curves. These families are of considerable importance, as they are ideal for viewing effect of multiple zeros on nearby zeros. By work of C. Breuil, B. Conrad, F. Diamond. R. Taylor and A. Wiles \cite{BCDT,TW,Wi}, the $L$-function of an elliptic curve agrees with that of a weight 2 cuspidal newform of level $N$ (where the integer $N>1$ is the conductor of the elliptic curve). There are many similarities between these $L$-functions and that associated to the Ramanujan tau function, and two major differences. The first is that the tau function is associated to a weight 12 cusp form, and the second is that the level of the tau function is 1 and not $N$. Both of these effects make the tau function more amenable to analysis and numerical experimentation: the higher weight leads to less discretization in the value of the $L$-function at the central point, and the level being 1 means that there are no bad primes in the explicit formula. %The lower order terms in the one-level density of families of elliptic curves are playing a key role in understanding the behavior of the observed repulsion of zeros near the central point, and our results are the next step towards understanding them.

In spite of these differences, the analysis of our family of $L$-functions is comparable to that of the family of quadratic twists of the $L$-function associated to an elliptic curve, for which the Ratios Conjecture's predictions have not yet been shown to agree with the number theory results.  Analysis of the family of quadratic twists of the tau $L$-function provides a useful starting point for the elliptic curve-based family. The first lower order term of this family is very important in ongoing investigations of the excess repulsion observed in the first zero above the central point (see \cite{DHKMS1, DHKMS2}. The Ratios Conjecture's prediction for these lower order terms have been inputted in some of these models, but had yet to be verified as of the original writing of this thesis.  The verification, performed in \cite{HMM}, is comparable to, and in some cases was guided by, the work in Sections \ref{section2} and \ref{section3}. The analysis of quadratic twists of the Ramanujan tau function is almost identical to the analysis needed there, the only difference being the effects of the bad primes are not present. Thus this work provides the framework that can be applied to study these elliptic curve families.

\subsection{Random Matrix Theory}
Random Matrix Theory (see \cite{matrix,Dy1,Dy2,KaSa1,KaSa2,Wig1,Wig2,Wig3,Wig4,Wig5,Wis}) has been extraordinarily successful in modeling diverse systems ranging from nuclear physics to statistics to number theory. In this thesis we are interested in its applications to predicting the behavior of $L$-functions. The $n$-level correlation between the normalized zeros of the Riemann zeta-function and the normalized eigenvalues of matrices in the Gaussian Unitary Ensemble was first noted in the early 1970s by Dyson and Montgomery \cite{Mon}, and then extended by many others (see \cite{Hej,Od1,Od2,RS}). While the behavior of zeros far from the central point is universal, he behavior near the central point depends on the family. This is observed in additional statistics such as $n$-level densities (see for example \cite{ILS,KaSa1,KaSa2}) and moments (see for example \cite{CFKRS}). Following the success of these investigations, Random Matrix Theory has served as an extremely useful tool for predicting the behavior of $L$-functions.

In an attempt to model the Rankin-Selberg convolution of families of matrices, we investigate the eigenvalue statistics of Kronecker products of matrices in the Gaussian Unitary Ensemble.  Inspired by \cite{DM2}, we look at lowest eigenangle statistics of combinations of different types of matrices to see if there seems to be a multiplicative symmetry constant.  Qualitative attributes of our computed distributions indicate that this is not the appropriate model for convolving families of $L$-functions.  For instance, orthogonal combined with orthogonal looks symplectic on the number theory side; but the distribution of lowest eigenangles for the orthogonal/orthogonal matrix combination features repulsion from zero, while that for symplectic matrices does not.  However, the similarities and differences between various combinations suggest that there is a great deal of structure in the eigenangle statistics of these Kronecker products that warrants further investigation.

\newpage

\section{Quadratic Twists of the Tau $L$-function}\label{section2}

The first family of $L$-functions used in our main convolution is $$\mathcal{G} \ = \ \{L(s,\chi_d)\,|\,\text{$d>0$ is an even fundamental discriminant}\},$$
where $\chi_d$ is the quadratic Dirichlet character modulo $d$.  A Dirichlet character (modulo $d$), denoted $\chi$, is a type completely multiplicative function on the units of $\mathbb{Z}$ with period $d$, and $\chi_d$ denotes the unique quadratic Dirichlet character $\mod d$.  We let $d$ be a fundamental discriminant, meaning that either $d\equiv 1\mod 4$ is square-free or $d/4\equiv 2,3\mod 4$ is square-free. We further restrict to $d$ even.  If $\chi_d$ is the quadratic character associated to the fundamental discriminant $d$ with $d>0$, we have $\chi_d(-1)=1$.

The second family, which consists of one element, arises from Ramanujan's tau function.  The Ramanujan tau function $\tau:\mathbb{N}\rightarrow\mathbb{Z}$ is defined by the coefficients of the Fourier expansion of $\eta(\tau)^{24}$, where $\eta$ is the Dedekind eta function.  That is,
\be\eta(z)^{24} \ = \ \sum_{n=1}^\infty \tau(n)q^n
\ee with $q=e^{2\pi i z}$.  Note that $\eta(z)^{24}$ is a scalar multiple of the discriminant modular function, a holomorphic cusp form of weight $12$ and level $1$.  In 1917 Mordell proved that $\tau(mn)=\tau(m)\tau(n)$ if $\gcd(m,n)=1$ (that is, $\tau$ is a multiplicative function) and that
\be
\tau(p^{r+1}) \ = \ \tau(p)\tau(p)^r-p^{11}\tau(p^{r-1})\label{tauid}
\ee for $p$ prime and $r$ a positive integer.  In 1974 Deligne proved that $|\tau(p)|\leq 2p^{11/2}$ for all $p$ prime.  (For more on the tau function see \cite{Serre}.) Defining $\tau^*(n)=\tau(n)/n^{11/2}$ , we have $|\tau^*(p)|\leq 2$ for all $p$ prime.  Using equation \eqref{tauid}, we have
\begin{align}
\tau^*(p^{r+1}) \ = \ &\frac{\tau(p^{r+1})}{p^{(r+1)(11/2)}}\nonumber
\\ \ = \ &\frac{\tau(p)\tau(p)^r}{p^{(r+1)(11/2)}}-\frac{p^{11}\tau(p^{r-1})}{p^{(r+1)(11/2)}}\nonumber
\\ \ = \ &\frac{\tau(p)}{p^{11/2}}\frac{\tau(p)^r}{p^{r(11/2)}}+\frac{\tau(p^{r-1})}{p^{(r-1)(11/2)}}\nonumber
\\ \ = \ &\tau^*(p)\tau^*(p^r)-\tau^*(p^{r-1})
\end{align} for $p$ prime and $r$ a positive integer.  Since  $\tau^*$ is a multiplicative function, we may consider the $L$-function \be L(s,\tau^*)\ = \ \sum_{n=1}^\infty\frac{\tau^*(n)}{n^s}\ = \ \prod_{\text{$p$ prime}}\left(1-\frac{\tau^*(p)}{p^s}+\frac{1}{p^{2s}}\right)^{-1} \ee for $Re(s)>1$.

We consider the $L$-function families $\mathcal{G}=\{L(s,\chi_d)\,|\text{$\chi_d$ is a quadratic character}\,\}$ and $\mathcal{H}=\{L(s,\tau^*)\}$ (noting that $\mathcal{H}$ has only one element).  Convolving these families, we have (by the work of Due\~nez-Miller \cite{DM2}) the orthogonal family $\mathcal{F}=\mathcal{G}\times\mathcal{H}$ of quadratic twists of the $L$-function $L(s,\tau^*)$, denoted by $L_\Delta(s,\chi_d)$.  The Ratios Conjecture's calculations for this family of $L$-functions were performed by Conrey and Snaith in \cite{CS1}.  To test the power of the Ratios Conjecture as it applies to the convolution $\mathcal{F}$, we perform the number theory computations and determine the one-level density of the zeros for suitably restricted test functions. This comprises the remainder of Section \S\ref{section2}. We then compare this to the Ratios Conjecture's predictions  in Section \S\ref{section3}, and see that they agree up to $O(X^{-(1-\sigma)/2+\varepsilon})$, where the support of the transform of our test function is contained in $(-\sigma,\sigma)$, where $\sigma<1$ (i.e., $\text{supp}(\hphi)\subset(-\sigma,\sigma)$).

\subsection{The Explicit Formula}

In this subsection we derive the explicit formula, which connects sums of our test function evaluated at the zeros of our family to sums of the Fourier transform of our test function evaluated at the logarithms of the primes; the one-level density is just a scaled version of this. We follow the arguments in \cite{RS}.

Let $L_\Delta(s,\chi_d)\in\mathcal{F}$. The essence of our strategy is to consider a contour integral of the logarithmic derivative $L_\Delta(s,\chi_d)$ and then shift this integral, picking up contributions from the zeros of $L_\Delta(s,\chi_d)$.  As $L_\Delta(s,\chi_d)$ appears in the denominator of this logarithmic derivative, the contour shift of this integral picks up those zeros as poles, giving us information about their distribution.  We analyze the resulting expression for a fixed $d$ and then take the limit of the average over all $d\leq X$ (as we cannot average over an infinite number).  For the purposes of averaging, we define $X^*=\sum_{d\leq X}1$ where $d$ is an even fundamental discriminant. By Lemma \ref{firstmillerlemma} we have \be X^\ast \ = \ \frac{3}{\pi^2}X + O(X^{1/2}), \ee and thus $X^\ast$ is of the same order of magnitude as $X$. For all subsequent sums over $d$, this will be the range of $d$ (i.e., we always assume $d$ to be an even fundamental discriminant at most $X$).

First we establish some key formulas.  Written as an Euler sum and an Euler product, we have
\begin{align}L_\Delta(s,\chi_d) \ = \ \sum_{n=1}^\infty \frac{\chi_d(n)\tau^*(n)}{n^s} \ = \ &\prod_{p}\left(1-\frac{\tau^*(p)\chi_d(p)}{p^s}+\frac{\chi_d(p^2)}{p^{2s}}\right)^{-1}\nonumber
\\ \ = \ &\prod_p\left(1-\frac{\alpha_p\chi_d(p)}{p^s}\right)^{-1}\left(1-\frac{\overline{\alpha}_p\chi_d(p)}{p^s}\right)^{-1}
\label{productform}
\end{align}
where $\alpha_p,\overline{\alpha}_p$ are the roots of the quadratic (in $1/{p^s}$) equation $1-{\tau^*(p)\chi_d(p)}/{p^s}+{\chi_d(p^2)}/{p^{2s}}$, meaning they are $\left({{\tau^*(p)\chi_d(p)\pm\sqrt{(\tau^*(p)\chi_d(p))^2-4\chi_d(p^2)}}}\right)/{\chi_d(p^2)}$.  Given that 
$$\tau^*(p)\chi_d(p))^2-4\chi_d(p^2)\leq 0,$$ these roots are either the same (and real) or are distinct and complex conjugates of one another.  In both cases, we have that they are complex conjugates (justifying our notation), and that they  satisfy $\alpha_p\cdot\overline{\alpha}_p=1$ and $\alpha_p+\overline{\alpha}_p=\tau^*(p)$.  Since both have multiplicative inverse equal to complex conjugate, they are both of norm $1$.    We now wish to extend our function to the entire complex plane.  For $d>0$, our $L$-function has the functional equation
\begin{equation}
\xi_\Delta(s,\chi_d):=\left(\frac{d}{2\pi}\right)^s\Gamma(s+11/2)L_\Delta(s,\chi_d) \ = \ \xi_\Delta(1-s,\chi_d)\label{extended}\ \end{equation}
(see, for instance, \cite{CS1,IK}).

We integrate the logarithmic derivative of $\xi_\Delta(s,\chi_d)$ weighted by a Schwartz function to ensure sufficient decay rate. We assume the Generalized Riemann Hypothesis (GRH), so that  if $\frac{1}{2}+i\gamma$ is a zero of $\xi(s,\chi_d)$ then $\gamma\in\mathbb{R}$.  Let $\phi$ be an even Schwartz function where its Fourier transform
\begin{align}
\hphi(\omega) \ = \ \int_{-\infty}^\infty\phi(x)e^{2\pi ix\omega}dx\label{ftphi}
\end{align}
has finite support;  that is, $\text{supp}(\hphi)\subset(-\sigma,\sigma)$ for some finite $\sigma$.    Extend $\phi(x)$ to the whole complex plane via
\begin{align}
H(s) \ = \ \phi\left(\frac{s-\frac{1}{2}}{i}\right)\label{definitionh}
.\end{align}
Note that $H(s)$ is scaled so that if $s=\frac{1}{2}+i\gamma_d$ is a zero of $\xi(s,\chi_d)$, $H(s)=\phi(\gamma_d)$.
Set
\begin{align}
I \ = \ \frac{1}{2\pi i}\int_{Re(s)=3/2}\frac{\xi_\Delta'(s,\chi_d)}{\xi_\Delta(s,\chi_d)}H(s)ds\label{intervaldef}
\end{align}
Shifting the contour to $Re(s)=-\frac{1}{2}$, we have that the only contribution is from the zeros of $\xi_\Delta(s,\chi_d)$ (which are the poles of the integrand), giving us

\begin{align}
I \ = \ \sum_{\gamma_d}\phi(\gamma_d)+\frac{1}{2\pi i}\int_{Re(s)=-1/2}\frac{\xi_\Delta'(s,\chi_d)}{\xi_\Delta(s,\chi_d)}H(s)ds\label{intervalwithzeros}
\end{align}
where $\gamma_d$ is the imaginary part of a non-trivial zero, and the sum is over all such values.  Recall from equation \eqref{extended} that $\xi_\Delta(s,\chi_d)=\xi_\Delta(1-s,\chi_d)$; it follows that $\xi_\Delta'(s,\chi_d)=-\xi_\Delta'(1-s,\chi_d)$.  Combined with equation \eqref{intervalwithzeros}, this gives us

\begin{align}
I \ = \ \sum_{\gamma_d}\phi(\gamma_d)-\frac{1}{2\pi i}\int_{Re(s)=-1/2}\frac{\xi_\Delta'(1-s,\chi_d)}{\xi_\Delta(1-s,\chi_d)}H(s)ds\label{shiftedinterval}
.\end{align}
  Performing the change of variables $s\rightarrow1-s$, we obtain
  \begin{align}
  I \ = \ \sum_{\gamma_d}\phi(\gamma_d)-\frac{1}{2\pi i}\int_{Re(s)=3/2}\frac{\xi_\Delta'(s,\chi_d)}{\xi_\Delta(s,\chi_d)}H(1-s)ds\label{intervalcov}
  .\end{align}
  Subtracting equation \eqref{intervalcov}  from  \eqref{intervaldef} proves
  \begin{theorem}\label{theoremzeros}
\begin{align}
\sum_{\gamma_d}\phi(\gamma_d) \ = \ \frac{1}{2\pi i}\int_{Re(s)=3/2}\frac{\xi_\Delta'(s,\chi_d)}{\xi_\Delta(s,\chi_d)}[H(s)+H(1-s)]ds\label{yayzeros}.\end{align}
\end{theorem}

This result, when properly averaged over a finite subset of the family $\mathcal{F}$, will give us the one-level density.

\subsection{Analyzing the Sum Over  Zeros}

Having found an expression for $\sum_\gamma\phi(\gamma)$ for a fixed $d$, we wish to manipulate it into a more informative form before averaging over $d$ to obtain the one-level density for our family.  First we find more a more useful way to express the logarithmic derivative of $\xi_\Delta(s,\chi_d)$.  Taking the logarithmic derivative of equation \eqref{extended}, we have
\begin{align}
\frac{\xi_\Delta'(s,\chi_d)}{\xi_\Delta(s,\chi_d)} \ = \ \log\left(\frac{d}{2\pi}\right)+\frac{\Gamma'(s+11/2)}{\Gamma(s+11/2)}+\frac{L_\Delta'(s,\chi_d)}{L_\Delta(s,\chi_d)}\label{extendedld}
.\end{align}
It will also be useful to have the logarithmic derivative of equation \eqref{productform}, which is
\begin{align}
\frac{L_\Delta'(s,\chi_d)}{L_\Delta(s,\chi_d)} \ = \ &-\sum_p\log p\left(\frac{\frac{\alpha_p\chi_d(p)}{p^s}}{1-\frac{\alpha_p\chi_d(p)}{p^s}}+\frac{\frac{\overline{\alpha}_p\chi_d(p)}{p^s}}{1-\frac{\overline{\alpha}_p\chi_d(p)}{p^s}}\right)\nonumber
\\ \ = \ &-\sum_{k=1}^\infty\sum_p\log p\frac{(\alpha_p^k+\overline{\alpha}_p^k)\chi_d^k(p)}{p^{sk}}\label{productformld}
.\end{align}

Using equation \eqref{extendedld} we expand the logarithmic derivative in \eqref{yayzeros} and shift the contours of all terms except the $\frac{L_\Delta'(s,\chi_d)}{L_\Delta(s,\chi_d)}$ term to $Re(s)=\frac{1}{2}$.  This gives us
\begin{align}
\sum_{\gamma_d}\phi(\gamma_d) \ = \ I_1+I_2\label{intervalsplit}
\end{align}
where
\begin{align}
I_1 \ = \ \frac{1}{2\pi i}\int_{Re(s)=1/2}\left[\log\left(\frac{d}{2\pi}\right)+\frac{\Gamma'}{\Gamma}(s+11/2)\right][H(s)+H(1-s)]ds\label{intervalone}
\end{align}
and
\begin{align}
I_2 \ = \ \frac{1}{2\pi i}\int_{Re(s)=3/2}\frac{L_\Delta'(s,\chi_d)}{L_\Delta(s,\chi_d)}[H(s)+H(1-s)]ds\label{intervaltwo}
.\end{align}
The integral in \eqref{intervalone} with $s=\frac{1}{2}+iy$ is
\begin{align}
I_1 \ = \ &\frac{1}{2\pi i}\int_{-\infty}^\infty\left[\log\left(\frac{d}{2\pi}\right)+\frac{\Gamma'}{\Gamma}\left(\frac{1}{2}+iy+\frac{11}{2}\right)\right]2\phi(y)idy\nonumber
\\ \ = \ &\frac{1}{2\pi }\int_{-\infty}^\infty\left[2\log\left(\frac{d}{2\pi}\right)+2\frac{\Gamma'}{\Gamma}\left(6+iy\right)\right]\phi(y)dy\nonumber
\\ \ = \ &\frac{1}{2\pi }\int_{-\infty}^\infty\left[2\log\left(\frac{d}{2\pi}\right)+\frac{\Gamma'}{\Gamma}\left(6+iy\right)+\frac{\Gamma'}{\Gamma}\left(6-iy\right)\right]\phi(y)dy\label{intonesimp}
.\end{align}

We now analyze $I_2$.  Combining equations \ref{productformld} and \ref{intervaltwo}, we have
\begin{align}
I_2 \ = \ &-\frac{1}{2\pi i}\int_{Re(s)=3/2}\sum_{k=1}^\infty\sum_p\frac{(\alpha_p^k+\overline{\alpha}_p^k)\chi_d^k(p) \log p}{p^{ks}}[H(s)+H(1-s)]ds\label{inttworewritten}.\end{align}

We wish to switch the order of integration and summation (over $k$ and $p$).  To justify this, we will prove

\begin{lemma}
\begin{align}
\int_{Re(s) =  3/2}\sum_{k=1}^\infty\sum_p\left|\frac{(\alpha_p^k+\overline{\alpha}_p^k)\chi_d^k(p) \log p}{p^{ks}}[H(s)+H(1-s)]\right|ds<\infty\label{lessthaninfinity}.\end{align}
\end{lemma}

\begin{proof}
Note that
\begin{align}\sum_{k=1}^\infty\sum_p\left|\frac{(\alpha_p^k+\overline{\alpha}_p^k)\chi_d^k(p) \log p}{p^{ks}}\right|
\ \leq\ &\sum_{k=1}^\infty\sum_p\frac{\log{p}\cdot|\alpha^k_p+\overline{\alpha}_p^k|\cdot|\chi_d^k(p)|}{|p^{k(3/2+iy)}|}\nonumber
\\\ \leq\ &\sum_{k=1}^\infty\sum_p\frac{2\cdot 1 \cdot \log p}{p^{3k/2}\cdot|e^{iyk\log{p}}|}\nonumber
\\\ \leq\ &\sum_{k=1}^\infty\sum_p\frac{2\log{p}}{(p^k)^{3/2}}\nonumber
\\\ \leq\ &\sum_{n=1}^\infty\frac{2\log{n}}{n^{3/2}}\nonumber
\\ \ = \ &-2\zeta'(3/2),
\end{align} a constant independent of $s$ and $d$. Thus we have
\begin{align}
\int_{Re(s)=3/2}\sum_{k=1}^\infty&\sum_p\left|\frac{(\alpha_p^k+\overline{\alpha}_p^k)\chi_d^k(p) \log p}{p^{ks}}[H(s)+H(1-s)]\right|ds\nonumber
\\ \ = \ &\int_{Re(s)=3/2}|H(s)+H(1-s)|\sum_{k=1}^\infty\sum_p\left|\frac{(\alpha_p^k+\overline{\alpha}_p^k)\chi_d^k(p) \log p}{p^{ks}}\right|ds\nonumber
\\\ \leq\ &C\int_{Re(s)=3/2}|H(s)+H(1-s)|ds\nonumber
\\\ \leq\ &C\left(\int_{Re(s)=3/2}|H(s)|ds+\int_{Re(s)=3/2}|H(1-s)|ds\right)\label{jusths}
.\end{align}

To see that both integrals in equation \eqref{jusths} are convergent, note that \begin{align}
\phi(x) \ = \ &\int_{-\infty}^\infty\hphi(\omega)e^{2\pi i x\omega}d\omega\nonumber
\\\phi(x+iy) \ = \ &\int_{-\infty}^\infty\hphi(\omega)e^{2\pi i (x+iy)\omega}d\omega\nonumber
\\H(x+iy) \ = \ &\int_{-\infty}^\infty\left(\hphi(\omega)e^{2\pi (x-\frac{1}{2})\omega}\right)\cdot e^{2\pi iy\omega}d\omega\label{decayingh}
.\end{align}
%Note that $H(x+iy)$ is rapidly decreasing in $y$:  for a fixed $x$, $H(x+iy)$ is the Fourier transform of a nice function, allowing us to apply the Riemann-Lebesgue lemma.
For fixed $x$, $H(x+iy)$ is the Fourier transform of a Schwartz function (namely $\hphi(\omega)e^{2\pi (x-\frac{1}{2})}$), meaning that it itself is Schwartz.  This means that it decays faster than $1/{y^k}$ for any $k\in\mathbb{N}$, implying that both integrals converge. The claim (equation \eqref{lessthaninfinity}) follows.
\end{proof}

By the Fubini-Tonelli Theorem, we may switch summation and integration in \eqref{lessthaninfinity} as the absolute value leads to a finite integral in the product measure. Doing so, pulling out terms constant with respect to $s$, and noting that $1/{p^{ks}}=e^{-ks\log p}$, we may rewrite equation \eqref{inttworewritten} as
\begin{align}
I_2 \ = \ &-\frac{1}{2\pi i}\sum_{k=1}^\infty\sum_p (\alpha_p^k+\overline{\alpha}_p^k)\chi_d^k(p) \log p\int_{Re(s)=3/2}[H(s)+H(1-s)]e^{-ks\log p}ds\label{inttwoorderswitched}.\end{align}
We wish to shift our contour to $Re(s)=\frac{1}{2}$.  Consider the integral
\begin{align}
\int_C[H(s)+H(1-s)]e^{-ks\log p}ds\label{rectangleintegral}
\end{align}
where $C$ is the rectangle defined by the points $\frac{3}{2}+iM$, $\frac{3}{2}-iM$,  $\frac{1}{2}+iM$, and $\frac{1}{2}-iM$, where $M>0$. As there are no poles of our integrand, this integral equals $0$. (Note the original integrand did have poles from the zeros of the $L$-function; however, by switching the order of summation and integration and considering the integral for a fixed prime, we need only consider integrals of analytic functions.)  As $M\rightarrow\infty$, the horizontal components of the rectangle become negligible (since $H(x+iy)$ decays rapidly as $y$ increases), meaning that in the limit the two vertical components must cancel each other.  It follows that
\begin{align}\int_{Re(s)=3/2}[H(s)+H(1-s)]e^{-ks\log p}ds \ = \ \int_{Re(s)=1/2}[H(s)+H(1-s)]e^{-ks\log p}ds.
\end{align}

Shifting contours as described above and changing variables by $s=\frac{1}{2}+iy$, we have
\begin{align}
I_2 \ = \ &-\frac{1}{2\pi i}\sum_{k=1}^\infty\sum_p(\alpha_p^k+\overline{\alpha}_p^k)\chi_d^k(p)\log p\int_{-\infty}^\infty2\phi(y)e^{-k(1/2+iy)\log p}idy\nonumber
\\ \ = \ &-\frac{2}{2\pi }\sum_{k=1}^\infty\sum_p\frac{(\alpha_p^k+\overline{\alpha}_p^k)\chi_d^k(p)\log p}{p^{k/2}}\int_{-\infty}^\infty\phi(y)e^{-2\pi i y\frac{\log p^k}{2\pi}}dy\nonumber
\\ \ = \ &-\frac{2}{2\pi }\sum_{k=1}^\infty\sum_p\frac{(\alpha_p^k+\overline{\alpha}_p^k)\chi_d^k(p)\log p}{p^{k/2}}\hphi\left(\frac{\log p^k}{2\pi}\right)\label{inttwosimp}.
\end{align}
Combining equations \eqref{intonesimp} and \eqref{inttwosimp}, we have
\begin{align}
\sum_{\gamma_d}\phi(\gamma_d) \ = \ &\frac{1}{2\pi }\int_{-\infty}^\infty\left[2\log\left(\frac{d}{2\pi}\right)+\frac{\Gamma'}{\Gamma}\left(6+iy\right)+\frac{\Gamma'}{\Gamma}\left(6-iy\right)\right]\phi(y)dy\nonumber
\\&-\frac{2}{2\pi i}\sum_{k=1}^\infty\sum_p\frac{(\alpha_p^k+\overline{\alpha}_p^k)\chi_d^k(p)\log p}{p^{k/2}}\hphi\left(\frac{\log p^k}{2\pi}\right)\label{combinedsimps}
.\end{align}

To rewrite equation \eqref{combinedsimps}, we sum over twists $d$ and scale the zeros by the mean density of zeros arising from even fundamental discriminants at most $X$. One could instead consider the related quantities where each $L$-function's zeros are scaled by the logarithm of its conductor, a local instead of a global rescaling. Similar behavior is seen; see for example \cite{GM,Mil1} for such investigations.

We set
\begin{align}L \ = \ \log\left(\frac{X}{2\pi}\right)\label{definitionl}\end{align} (which is essentially the average log-conductor)
and replace $\phi(y)$ with
\begin{align}g(\nu) \ = \ \phi\left(y\right)\end{align}
where $\nu=y\cdot\frac{L}{\pi}$. It is a straightforward calculation that if $F(x)=G(cx)$ (where $c\neq 0$) and $\widehat{F}(\omega)$ is the Fourier transform of $F(x)$, then $\frac{1}{c}\widehat{F}\left(\frac{\omega}{c}\right)$ is the Fourier transform of $G(cx)$.  It follows that $\hat{\varphi}(\omega)=\frac{\pi}{L}\hat{g}\left(\frac{\pi}{L}\omega\right)$.  Summing over quadratic twists with even fundamental discriminant $d\leq X$ and dividing by $X^*$, the number of terms in the sum (which is proportional to $X$), we have proven a tractable explicit formula for the one-level density.

\begin{theorem}[Expansion for the one-level density]\label{thm:firsttractableexpansiononelevel} The one-level density for the family of twists of the Ramanujan tau function by even fundamental discriminants at most $X$ satisfies
\begin{align}
&\frac{1}{X^*}\sum_{d\leq X}\sum_{\gamma_d}g\left(\gamma_d\frac{L}{\pi}\right)\nonumber
\\ \ = \ &\frac{1}{2LX^* }\int_{-\infty}^\infty g(\nu)\sum_{d\leq{X}}\left[2\log\left(\frac{d}{2\pi}\right)+\frac{\Gamma'}{\Gamma}\left(6+i\frac{\pi\nu}{L}\right)+\frac{\Gamma'}{\Gamma}\left(6-i\frac{\pi\nu}{L}\right)\right]d\nu\nonumber
\\&-\frac{2}{2LX^*}\sum_{d\leq X}\sum_{k=1}^\infty\sum_p\frac{(\alpha_p^k+\overline{\alpha}_p^k)\chi_d^k(p)\log p}{p^{k/2}}\hat{g}\left(\frac{\log p^k}{2L}\right)\label{averaged},
\end{align} where $\phi$ is an even Schwartz function such that $\supp(\hphi)$ is contained in a bounded interval.
\end{theorem}

\subsection{Analyzing the One-Level Density}

We analyze the term
\begin{align}S \ = \ -\frac{2}{2LX^*}\sum_{d\leq X}\sum_{k=1}^\infty\sum_p\frac{(\alpha_p^k+\overline{\alpha}_p^k)\chi_d^k(p)\log p}{p^{k/2}}\hat{g}\left(\frac{\log p^k}{2L}\right).\label{sdefinition}
\end{align}
by splitting it into two sums:
\begin{align}S \ = \ S_{\rm even}+S_{\rm odd}
\end{align}
Specifically,
\begin{align}
S_{\rm even} \ = \ -\frac{1}{X^*}\sum_{d\leq X}\sum_{k=1}^\infty\sum_p\frac{(\alpha_p^{2k}+\overline{\alpha}_p^{2k})\chi_d^{2}(p)\log p}{p^{k}L}\hat{g}\left(\frac{\log p^k}{L}\right)\label{sevendefinition}
\end{align}
and
\begin{align}
S_{\rm odd} \ = \ -\frac{1}{X^*}\sum_{d\leq X}\sum_{k=0}^\infty\sum_p\frac{(\alpha_p^{2k+1}+\overline{\alpha}_p^{2k+1})\chi_d(p)\log p}{p^{(2k+1)/2}L}\hat{g}\left(\frac{\log p^{2k+1}}{L}\right).\label{sodddefinition}
\end{align}
No higher powers of $\chi_d(p)$ appear because $\chi_d$ is a quadratic character, implying that $\chi_d^{2k}(p) = \chi_d^2(p)$ and $\chi_d^{2k+1}(p)=\chi_d(p)$ for any positive integer $k$.

Note that
\be
\chi_d^2(p)  \ = \
\begin{cases}
1 \mbox{~if~} p \nmid d,\\
0 \mbox{~if~} p|d.
\end{cases}
\ee
This allows us to split $S_{\rm even}$ into
\be S_{\rm even} \ = \ S_{\rm even,1}+S_{\rm even,2}
\ee
where
\be
S_{\rm even,1}  \ = \  -\frac{1}{L}\sum_p \sum_{k = 1}^\infty  \frac{(\alpha_p^{2k} + \overline{\alpha}_p^{2k}) \log p}{p^k } \widehat{g}\left(\frac{\log p^k}{L} \right)
\ee
and
\be
S_{\rm even,2}  \ = \  \frac{1}{X^*} \sum_{d \leq X}\sum_{k = 1}^\infty \sum_{p|d} \frac{(\alpha_p^{2k} + \overline{\alpha}_p^{2k}) \log p}{p^k L} \widehat{g}\left(\frac{\log p^k}{L} \right)
\ee
(there is no $1/X^\ast$ in $S_{{\rm even},1}$ as that was canceled by the $X^\ast$ from the $d$-sum). We will analyze these two terms separately.

 Consider $S_{\rm even,1}$.  By means of standard techniques as seen in \cite{Mil3} and \cite{HMM}, we obtain
 \bea S_{{\rm even};1} & \ =
\ &  \frac{g(0)}2 + \frac{1}{L}\intii
g(\nu) \left(\frac{L'}{L}\left(1+\frac{2\pi i \nu}{L},{\rm sym}^2\Delta\right) - \frac{\zeta'}{\zeta}\left(1+\frac{2\pi i \nu}{L}\right)\right) d\nu
\nonumber\\ \eea
(A detailed proof of this can be found Appendix \ref{sec:appendixkeylemmas}.)

Consider $S_{\rm even,2}$.  Changing the order of summation, we may write
\begin{align}S_{\rm even,2}  \ = \ & \frac{1}{X^*} \sum_{d \leq X}\sum_{k = 1}^\infty \sum_{p|d} \frac{(\alpha_p^{2k} + \overline{\alpha}_p^{2k}) \log p}{p^k L} \widehat{g}\left(\frac{\log p^k}{L} \right)\nonumber
\\ \ = \ & \frac{1}{LX^*} \sum_{p} \sum_{k = 1}^\infty \frac{(\alpha_p^{2k} + \overline{\alpha}_p^{2k}) \log p}{p^k } \widehat{g}\left(\frac{\log p^k}{L} \right)\sum_{\substack{d \leq X \\ p|d}}1\label{sumofone}
.\end{align}
By Lemma \ref{firstmillerlemma} (proven in Appendix \ref{sec:appendixkeylemmas}), we have
\be
\sum_{\substack{d \leq X \\ p|d}}1  \ = \ \frac{X^*}{p+1}+O(X^{1/2})\label{sumofone2}.
\ee 
Plugging \eqref{sumofone2} into \eqref{sumofone} yields
\be S_{\rm even,2} \ = \ \frac{1}{L} \sum_{p} \sum_{k = 1}^\infty  \frac{(\alpha_p^{2k} + \overline{\alpha}_p^{2k}) \log p}{p^k(p+1) } \widehat{g}\left(\frac{\log p^k}{L} \right)+O(X^{-1/2})\label{loglog},
\ee where we used Lemma \ref{firstmillerlemma} to note that $X^\ast = 3X/\pi^2 + O(X^{1/2})$. To see that the error term is $P(X^{-1/2})$, note the error term is bounded by
\begin{align}\frac{O(X^{1/2})}{LX^*} \sum_{p} \sum_{k = 1}^\infty \frac{\left|\alpha_p^{2k} + \overline{\alpha}_p^{2k}\right| \log p}{p^k } \widehat{g}\left(\frac{\log p^k}{L} \right)
.\end{align} As remarked, by Lemma \ref{firstmillerlemma} we have $O(X^{1/2}/X^*)  = O(X^{-1/2})$.  We  next note that the sum over $k \ge 2$ is trivially seen to be $O(1)$, and by Mertens' theorem (which states $\sum_{p \le X^\sigma} \log p / p = \log X^\sigma + O(1)$), the contribution from $k=1$ divided by $L$ (which is of size $\log X$) is also $O(1)$. This completes the proof of the size of the error term in \eqref{loglog} for $S_{\rm even,2}$. 

We now turn to the analysis of the main term of $S_{\rm even,2}$ in \eqref{loglog}. Note that
\be
\widehat{g}\left(\frac{\log p^k}{L} \right) \ = \ \int_{-\infty}^\infty g(\nu)e^{-2\pi i\frac{\log p^k}{L}}d\nu=\int_{-\infty}^\infty g(\nu)p^{\frac{-2\pi i\nu}{L}k}d\nu.\label{undofourier}
\ee
Combining \eqref{loglog} and \eqref{undofourier}, we have
\begin{align} \nonumber
S_{\rm even,2} &  \ = \  \frac{1}{L} \sum_{p}\sum_{k = 1}^\infty \frac{(\alpha_p^{2k} + \overline{\alpha}_p^{2k}) \log p}{p^k(p + 1)} \int_{-\infty}^\infty g(\nu)p^{-\frac{2\pi i \nu}{L}k}d\nu + O(X^{-1/2}) \\ \nonumber
&  \ = \  \frac{1}{L} \sum_{p}\sum_{k = 1}^\infty \frac{(\alpha_p^{2k} + \overline{\alpha}_p^{2k}) \log p}{p^k(p + 1)} \int_{-\infty}^\infty g(\nu)p^{-\frac{2\pi i \nu}{L}k}d\nu + O(X^{-1/2}) \\
&  \ = \  \frac{1}{L} \int_{-\infty}^\infty g(\nu) \sum_{p} \frac{\log p}{(p + 1)} \sum_{k = 1}^\infty \frac{(\alpha_p^{2k} + \overline{\alpha}_p^{2k}) }{p^{k(1 + \frac{2 \pi i \nu}{L})}} d\nu + O(X^{-1/2}).
\end{align}

We will now bound $S_{\rm odd}$.  The following lemma and the proof thereof were modified with permission from \cite{Mil3}.

\begin{lemma} For $\text{supp}(\hat{g})\subset (-\sigma,\sigma)$, we have $S_{{\rm odd}} = O(X^{-\frac{1-\sigma}2}
\log^6 X)$.
\end{lemma}

\begin{proof} Jutila's bound (see equation (3.4) of \cite{Ju3}) is \be \sum_{1 < n \le N
\atop n\ {\rm non-square}} \ \left| \sum_{0 < d \le X \atop d\ {\rm
fund.\ disc.}}\ \chi_d(n)\right|^2 \ \ll \ N X \log^{10} N \ee where
the $d$-sum, per usual, is over even fundamental discriminants at most $X$. As
$2k+1$ is odd, $p^{2k+1}$ is never a square. The bound above is of non-negative numbers summed over all non-squares; as our sum is over a subset, Jutila's bound holds for us as well. Thus Jutila's
bound gives \be \left(\sum_{\ell=0}^\infty \sum_{p^{(2\ell+1)/2} \le
X^\sigma} \left| \sum_{d \le X} \chi_d(p)\right|^2 \right)^{1/2} \
\ll \ X^{\frac{1+\sigma}2} \log^5 X. \ee
 Recall
\bea S_{{\rm odd}}
&\ =\ & -\frac{1}{X^\ast} \sum_{k=0}^\infty \sum_p \frac{(\alpha_p^{2k+1}+\overline{\alpha}_p^{2k+1})\log
p}{p^{(2k+1)/2} L}\ \hg\left(\frac{\log p^{2k+1}}{L}\right) \sum_{d\le X} \chi_d(p), \eea with $|\alpha_p^{2k+1}+\overline{\alpha}_p^{2k+1}| \le 2$.
Applying the Cauchy-Schwartz inequality and pulling out 2 for $|\alpha_p^{2k+1}+\overline{\alpha}_p^{2k+1}|$, we have \bea |S_{{\rm odd}}| &\ \le \ &
\frac2{X^\ast}\left(\sum_{k=0}^\infty \sum_{p^{2k+1} \le
X^\sigma} \left|\frac{\log p}{p^{(2k+1)/2} L}\
\hg\left(\frac{\log p^{2k+1}}{L}\right)\right|^2\right)^{1/2} \nonumber\\ & & \ \ \cdot \
\left(\sum_{k=0}^\infty \sum_{p^{2k+1} \le X^\sigma}
\left|\sum_{d\le X} \chi_d(p)\right|^2\right)^{1/2} \nonumber\\ &\ll
& \frac2{X^\ast}\left(\sum_{n \le X^\sigma} \frac1{n}\right)^{1/2}
\cdot X^{\frac{1+\sigma}2} \log^5 X \nonumber\\ &\ll &
X^{-\frac{1-\sigma}2}  \log^6 X. \eea
\end{proof}

As $\sigma>0$, this is a larger error term than the $O(X^{-1/2})$ we have from $S_{\rm{even}}$, and thus absorbs that error term. Taking all these pieces together, we find that the number theoretic calculations of the one-level density give us
\begin{align}
&\frac{1}{X^*}\sum_{d\leq X}\sum_{\gamma_d}g\left(\gamma_d\frac{L}{\pi}\right)\nonumber
\\= \ &\frac{1}{2LX^* }\int_{-\infty}^\infty g(\nu)\sum_{d\leq{X}}\left[2\log\left(\frac{d}{2\pi}\right)+\frac{\Gamma'}{\Gamma}\left(6+i\frac{\pi\nu}{L}\right)+\frac{\Gamma'}{\Gamma}\left(6-i\frac{\pi\nu}{L}\right)\right]d\nu\nonumber
\\&+\frac{g(0)}{2}+\frac{1}{L}\intii
g(\nu) \left(\frac{L'}{L}\left(1+\frac{2\pi i \nu}{L},{\rm sym}^2\Delta\right) - \frac{\zeta'}{\zeta}\left(1+\frac{2\pi i \nu}{L}\right)\right) d\nu\nonumber
\\&+ \frac{1}{L} \int_{-\infty}^\infty g(\nu) \sum_{p} \frac{\log p}{(p + 1)} \sum_{k = 1}^\infty \frac{(\alpha_p^{2k} + \overline{\alpha}_p^{2k}) }{p^{k(1 + \frac{2 \pi i \nu}{L})}} d\nu+ O(X^{-\frac{1-\sigma}2}
\log^6 X)  \nonumber\\
=\ &\frac{g(0)}{2}+\frac{1}{2LX^* }\int_{-\infty}^\infty g(\nu)\left(\sum_{d\leq {X}}\left[2\log\left(\frac{d}{2\pi}\right)+\frac{\Gamma'}{\Gamma}\left(6+i\frac{\pi\nu}{L}\right)+\frac{\Gamma'}{\Gamma}\left(6-i\frac{\pi\nu}{L}\right)\right]\right.\nonumber
\\&+2\left(\frac{L'}{L}\left(1+\frac{2\pi i \nu}{L},{\rm sym}^2\Delta\right) - \frac{\zeta'}{\zeta}\left(1+\frac{2\pi i \nu}{L}\right)\right.
\left.\left.+ \sum_{p} \frac{\log p}{(p + 1)} \sum_{k = 1}^\infty \frac{(\alpha_p^{2k} + \overline{\alpha}_p^{2k}) }{p^{k(1 + \frac{2 \pi i \nu}{L})}} \right)\right)d\nu \nonumber
\\&+ O(X^{-\frac{1-\sigma}2}\log^6 X)\label{ntpred}
.\end{align}

To have a power savings in the error term, we require $\sigma<1$.  Thus we have proven the first part of Theorem \ref{thm:mainnumbequalsratios}, namely that equation \eqref{ntpred} is the one-level density for the family of quadratic twists of the tau $L$-function for suitably restricted test functions.

Having calculated the one-level density on the number theory side, we now compare it to the Ratios Conjecture's predictions for the one-level density.

\newpage

\section{Comparison With the Ratios Conjectures' Predictions}\label{section3}

Using the Ratios Conjecture, Conrey and Snaith \cite{CS1} compute the one-level density for our family to be
\begin{align}
\sum_{d\leq X}\sum_{\gamma_d}\phi\left(\gamma_d\right)\ = \ &\frac{1}{2\pi }\int_{-\infty}^\infty \phi(y)\left(\sum_{d\leq X}\left[2\log\left(\frac{d}{2\pi}\right)+\frac{\Gamma'}{\Gamma}\left(6+iy\right)+\frac{\Gamma'}{\Gamma}\left(6-iy\right)\right]\right.\nonumber
\\&+2\left(-\frac{\zeta'}{\zeta}(1+2iy)+\frac{L'_\Delta}{L_\Delta}({\rm sym^2},1+2iy)+B'_\Delta(iy;iy)\right.\nonumber
\\&\left.\left.-\left(\frac{d}{2\pi}\right)^{-2iy}\frac{\Gamma(6-iy)}{\Gamma(6+iy)}\frac{\zeta(1+2iy)L_\Delta({\rm sym^2},1-2iy)}{L_\Delta({\rm sym^2},1)}B_\Delta(-iy,iy)\right)\right)dy \nonumber
\\&+ O(X^{1/2+\epsilon})\label{rcpredunscaled}
\end{align}where
\be L_\Delta\left({\rm sym^2},s\right)=\prod_p\left(1-\frac{\alpha^2_p}{p^s}\right)^{-1}\left(1-\frac{1}{p^s}\right)^{-1}\left(1-\frac{\overline{\alpha}^2_p}{p^s}\right)^{-1},\label{symdef}
\ee
\begin{align}
B_{\Delta}(\alpha;\gamma)=\prod_p&\left(1+\frac{p}{p+1}\left(\sum_{m=1}^\infty\frac{\tau^*(p^{2m})}{p^{m(1+2\alpha)}}\right.\right.\nonumber
\\-&\left.\left.\frac{\tau^*(p)}{p^{1+\alpha+\gamma}}\sum_{m=0}^\infty\frac{\tau^*(p^{2m+1})}{p^{m(1+2\alpha)}}+\frac{1}{p^{1+2\gamma}}\sum_{m=0}^\infty\frac{\tau^*(p^{2m})}{p^{m(1+2\alpha)}} \right) \right)\nonumber
\\\times &\frac{\left(1-\frac{\tau^*(p^2)}{p^{1+2\alpha}}+\frac{\tau^*(p^2)}{p^{2+4\alpha}}-\frac{1}{p^{3+6\alpha}}\right)\left(1-\frac{1}{p^{1+2\gamma}}\right)}{\left(1-\frac{\tau^*(p^2)}{p^{1+\alpha+\gamma}}+\frac{\tau^*(p^2)}{p^{2+2\alpha+2\gamma}}-\frac{1}{p^{3+3\alpha+3\gamma}}\right)\left(1-\frac{1}{p^{1+\alpha+\gamma}}\right)},
\end{align}
and where the derivative of $B_\Delta$ is with respect to $\alpha$.
Again setting $\phi(y)=g(\nu)$ and dividing by ${X^*}$, this equation becomes
\begin{align}
&  \frac{1}{X^*}\sum_{d\leq X}\sum_{\gamma_d}g\left(\gamma_d\frac{L}{\pi}\right)\ = \ \nonumber
\\& \frac{1}{2LX^* }\int_{-\infty}^\infty g(\nu)\left(\sum_{d\leq X}\left[2\log\left(\frac{d}{2\pi}\right)+\frac{\Gamma'}{\Gamma}\left(6+i\frac{\pi\nu}{L}\right)+ \frac{\Gamma'}{\Gamma}\left(6-i\frac{\pi\nu}{L}\right)\right]\right.\nonumber
\\&\ \ \ +\ 2\left(-\frac{\zeta'}{\zeta}\left(1+2i\frac{\pi\nu}{L}\right)+\frac{L'_\Delta}{L_\Delta}\left({\rm sym^2},1+2i\frac{\pi\nu}{L}\right)+B'_\Delta\left(i\frac{\pi\nu}{L};i\frac{\pi\nu}{L}\right)\right.\nonumber
\\&\left.\left.\ \ \ -\ \left(\frac{d}{2\pi}\right)^{-2i\frac{\pi\nu}{L}}\frac{\Gamma\left(6-i\frac{\pi\nu}{L}\right)}{\Gamma\left(6+i\frac{\pi\nu}{L}\right)}\frac{\zeta\left(1+2i\frac{\pi\nu}{L}\right)L_\Delta\left({\rm sym^2},1-2i\frac{\pi\nu}{L}\right)}{L_\Delta({\rm sym^2},1)}B_\Delta\left(-i\frac{\pi\nu}{L},i\frac{\pi\nu}{L}\right)\right)\right)d\nu \nonumber
\\&\ \ \ +\ O(X^{-1/2+\epsilon}).\label{rcpred}
\end{align}

We will show that equations \eqref{ntpred} and \eqref{rcpred} agree up $O(X^{-({1-\sigma})/2+\varepsilon})$, a power savings error term when $\sigma<1$.  As the two expressions agree in their general form, in the $2\log\left(d/2\pi\right)+\frac{\Gamma'}{\Gamma}\left(6+i\nu\right)+\frac{\Gamma'}{\Gamma}\left(6-i\nu\right)$ term, and in the $-\frac{\zeta'}{\zeta}\left(1+2i\frac{\pi\nu}{L}\right)+\frac{L'_\Delta}{L_\Delta}\left({\rm sym^2},1+2i\frac{\pi\nu}{L}\right)$ term, we need only analyze the two remaining terms of equation \eqref{rcpred}, showing the first one are equal to the remaining term in \eqref{ntpred} and bounding the second one as our error term (plus a constant $g(0)/2$ which balances the corresponding constant of the other equation).  We will frequently use the variable $y$ (equal to ${\pi\nu}/{L}$) for notational convenience.

\subsection{The Derivative of $B_\Delta$}

We will consider the $B'_\Delta(iy;iy)$ term, and show that it equals
\be
\sum_p\frac{\log p}{p+1}\sum_{m=1}^\infty\frac{\alpha_p^{2m}+\overline{\alpha}_p^{2m}}{p^{m(1+2i\nu)}}\label{equaltob}.
\ee Recall from \cite{CS1} that $B_\Delta(r;r)=1$ (as can be verified by direct substitution).  For notational convenience, we define
\begin{align}
f_1(\alpha;\gamma)=&1+\frac{p}{p+1}\left(\sum_{m=1}^\infty\frac{\tau^*(p^{2m})}{p^{m(1+2\alpha)}}-\frac{\tau^*(p)}{p^{1+\alpha+\gamma}}\sum_{m=0}^\infty\frac{\tau^*(p^{2m+1})}{p^{m(1+2\alpha)}}\right)+\frac{1}{p^{1+2\gamma}}\sum_{m=0}^\infty\frac{\tau^*(p^{2m})}{p^{m(1+2\alpha)}}
\\ f_2(\alpha;\gamma)=&1-\frac{\tau^*(p^2)}{p^{1+2\alpha}}+\frac{\tau^*(p^2)}{p^{2+4\alpha}}-\frac{1}{p^{3+6\alpha}}
\\f_3(\alpha;\gamma)=&1-\frac{1}{p^{1+2\gamma}}
\\f_4(\alpha;\gamma)=&1-\frac{\tau^*(p^2)}{p^{1+\alpha+\gamma}}+\frac{\tau^*(p^2)}{p^{2+2\alpha+2\gamma}}-\frac{1}{p^{3+3\alpha+3\gamma}}
 \\f_5(\alpha;\gamma)=&1-\frac{1}{p^{1+\alpha+\gamma}},
 \end{align} giving us
 \be B_{\Delta}(\alpha;\gamma)=\prod_p f_1(\alpha;\gamma)\cdot\frac{f_2(\alpha;\gamma)f_3(\alpha;\gamma)}{f_4(\alpha;\gamma)f_5(\alpha;\gamma)}\label{bsimpler}.
 \ee

 \begin{lemma}\bea
& &  \frac{\partial B}{\partial \alpha}(r;r)\ = \ \nonumber\\ & & \sum_p\log p\left(\frac{1}{p+1}\sum_{m=1}^\infty\frac{\alpha_p^{2m}+\overline{\alpha}_p^{2m}}{p^{m(1+2r)}}-\sum_{m=1}^\infty\frac{\alpha_p^{2m}+\overline{\alpha}_p^{2m}}{p^{m(1+2r)}}\right.\nonumber
\\& &\,\,\,\,\,\,\,+\left.\frac{\frac{\tau^*(p^2)}{p^{1+2r}}-\frac{2\cdot\tau^*(p^2)}{p^{2+4r}}+\frac{3}{p^{3+6r}}}{1-\frac{\tau^*(p^2)}{p^{1+2r}}+\frac{\tau^*(p^2)}{p^{2+4r}}-\frac{1}{p^{3+6r}}}+\frac{1}{1-p^{1+2r}}\right).
\nonumber\\
\eea
 \end{lemma}

 \begin{proof} By taking the logarithmic derivative, we reduce the product of our five functions to a sum, allowing us to more easily compute it piece by piece.
 Taking the logarithmic partial derivative of \eqref{bsimpler} with respect to $\alpha$, we have
 \begin{align}
 \frac{\frac{\partial B}{\partial\alpha}(\alpha;\gamma)}{B(\alpha,\gamma)}=&\frac{\partial}{\partial\alpha}\log\left(\prod_p f_1(\alpha;\gamma)\cdot\frac{f_2(\alpha;\gamma)f_3(\alpha;\gamma)}{f_4(\alpha;\gamma)f_5(\alpha;\gamma)}\right)\nonumber
 \\ \ = \ &\frac{\partial}{\partial\alpha}\sum_p\log\left(f_1(\alpha;\gamma)\cdot\frac{f_2(\alpha;\gamma)f_3(\alpha;\gamma)}{f_4(\alpha;\gamma)f_5(\alpha;\gamma)}\right)\nonumber
 \\ \ = \ &\frac{\partial}{\partial\alpha}\sum_p\left(\log(f_1(\alpha;\gamma))+\log(f_2(\alpha;\gamma))+\log(f_3(\alpha;\gamma))\right.\nonumber
 \\   &\left.\,\,\,\,\,\,\,\,\,\,\,\,\,\,\,\,\,\,-\log(f_4(\alpha;\gamma))-\log(f_5(\alpha;\gamma))\right)\nonumber
 \\ \ = \ &\sum_p\left(\frac{f'_1(\alpha;\gamma)}{f_1(\alpha;\gamma)}+\frac{f'_2(\alpha;\gamma)}{f_2(\alpha;\gamma)}+\frac{f'_3(\alpha;\gamma)}{f_3(\alpha;\gamma)}-\frac{f'_4(\alpha;\gamma)}{f_4(\alpha;\gamma)}-\frac{f'_5(\alpha;\gamma)}{f_5(\alpha;\gamma)}\right),
 \end{align}
 where the derivatives $f'_i$ are with respect to $\alpha$.  Since $B(r;r)=1$, we have
 \be
 \frac{\partial B}{\partial \alpha}(r;r)=\sum_p\left(\frac{f'_1(r;r)}{f_1(r;r)}+\frac{f'_2(r;r)}{f_2(r;r)}+\frac{f'_3(r;r)}{f_3(r;r)}-\frac{f'_4(r;r)}{f_4(r;r)}-\frac{f'_5(r;r)}{f_5(r;r)}\right).
 \ee
 We shall evaluate each of these logarithmic derivatives at $\alpha=\gamma=r$.

 Note that $f_1(r;r)=1$.  Taking the derivative of $f_1(\alpha;\gamma)$ with respect to $\alpha$, we have
 \begin{align}
 f'_1(\alpha;\gamma)\ = \ &\frac{p}{p+1}\left(\sum_{m=1}^\infty\frac{-2m\log p\cdot\tau^*(p^{2m})}{p^{m(1+2\alpha)}}-\sum_{m=0}^\infty\frac{-(2m+1)\log p\cdot\tau^*(p^{2m+1})\tau^*(p)}{p^{m(1+2\alpha)+1+\alpha+\gamma}}\right.\nonumber
 \\&+\left.\frac{1}{p^{1+2\gamma}}\sum_{m=0}^\infty\frac{-2m\log p\cdot\tau^*(p^{2m})}{p^{m(1+2\alpha)}} \right)\nonumber
 \\ \ = \ &-\frac{p\log p}{p+1}\left(\sum_{m=1}^\infty\frac{2m\tau^*(p^{2m})}{p^{m(1+2\alpha)}}-\sum_{m=0}^\infty\frac{(2m+1)\tau^*(p^{2m+1})\tau^*(p)}{p^{m(1+2\alpha)+1+\alpha+\gamma}}\right.\nonumber
 \\ & \,\,\,\,\,\,\,\,\,\,\,\,\,\,\,\,\,\,\,\,\,\,\,\,\,\,\,\,+\left.\frac{1}{p^{1+2\gamma}}\sum_{m=0}^\infty\frac{2m\tau^*(p^{2m})}{p^{m(1+2\alpha)}} \right).
 \end{align}
 Plugging in $\alpha=\gamma=r$, we have
 \begin{align}
 f_1'(r;r)=&-\frac{p\log p}{p+1}\left(\sum_{m=1}^\infty\frac{2m\tau^*(p^{2m})}{p^{m(1+2r)}}-\sum_{m=0}^\infty\frac{(2m+1)\tau^*(p^{2m+1})\tau^*(p)}{p^{(m+1)(1+2r)}}+\sum_{m=0}^\infty\frac{2m\tau^*(p^{2m})}{p^{(m+1)(1+2r)}} \right)\nonumber
 \\ \ = \ &-\frac{p\log p}{p+1}\sum_{m=1}^\infty\frac{2m\tau^*(p^{2m})+(2m-1)\tau^*(p^{2m-1})\tau^*(p)+(2m-2)\tau^*(p^{2m-2})}{p^{m(1+2r)}}.
 \end{align}  Using the fact that $\tau^*(p^{2m-1})\tau^*(p)=\tau^*(p^{2m})-\tau^*(p^{2m-2})$, this expression becomes
  \begin{align}
 f_1'(r;r)=&-\frac{p\log p}{p+1}\sum_{m=1}^\infty\frac{2m\tau^*(p^{2m})+(2m-1)(\tau^*(p^{2m})-\tau^*(p^{2m-2}))+(2m-2)\tau^*(p^{2m-2})}{p^{m(1+2r)}}\nonumber
 \\ \ = \ &-\frac{p\log p}{p+1}\sum_{m=1}^\infty\frac{\tau^*(p^{2m})-\tau^*(p^{2m-2})}{p^{m(1+2r)}}\nonumber
 \\ \ = \ &-\frac{p\log p}{p+1}\sum_{m=1}^\infty\frac{\alpha_p^{2m}+\overline{\alpha}_p^{2m}}{p^{m(1+2r)}}.
 \end{align}
Finally, noting that $-\frac{p}{p+1}=-1+\frac{1}{p+1}$, we have
\be f_1'(r;r)=\frac{\log p}{p+1}\sum_{m=1}^\infty\frac{\alpha_p^{2m}+\overline{\alpha}_p^{2m}}{p^{m(1+2r)}}-\log p\sum_{m=1}^\infty\frac{\alpha_p^{2m}+\overline{\alpha}_p^{2m}}{p^{m(1+2r)}}.
\ee

 We now move on to $f_2$ and $f_2'$.  Plugging in, we have
 \be f_2(r;r)\ = \ 1-\frac{\tau^*(p^2)}{p^{1+2r}}+\frac{\tau^*(p^2)}{p^{2+4r}}-\frac{1}{p^{3+6r}}.
 \ee
 Taking the derivative with respect to $\alpha$ and evaluating at $\alpha=\gamma=r$, we have
 \be f'_2(r;r)\ = \ \log p\left(\frac{2\cdot\tau^*(p^2)}{p^{1+2r}}-\frac{4\cdot\tau^*(p^2)}{p^{2+4r}}+\frac{6}{p^{3+6r}}\right)
 \ee
 Similar calculations give us
 \be f_4(r;r) \ = \ f_2(r,r)
 \ee
 and
 \be f'_4(r;r)\ = \ \log p\left(\frac{\tau^*(p^2)}{p^{1+2r}}-\frac{2\cdot\tau^*(p^2)}{p^{2+4r}}+\frac{3}{p^{3+6r}}\right).
 \ee  It follows that
 \be\frac{f_2'(r;r)}{f_2(r:r)}-\frac{f_4'(r;r)}{f_4(r:r)}\ = \ \log p\frac{\frac{\tau^*(p^2)}{p^{1+2r}}-\frac{2\cdot\tau^*(p^2)}{p^{2+4r}}+\frac{3}{p^{3+6r}}}{1-\frac{\tau^*(p^2)}{p^{1+2r}}+\frac{\tau^*(p^2)}{p^{2+4r}}-\frac{1}{p^{3+6r}}}.
 \ee

 Noting that $f_3(r;r)$ is a constant with respect to $\alpha$, we have $f'_3(r;r)=0$.  Finally, we have
 \be\frac{f_5'(r;r)}{f_5(r;r)}\
  = \ \frac{\log p\cdot\frac{1}{p^{1+2r}}}{1-\frac{1}{p^{1+2r}}}\ = \ \frac{\log p}{p^{1+2r}-1}.
 \ee
 Combining all these expressions, we have

  \begin{align}
   &  \frac{\partial B}{\partial \alpha}(r;r) \ = \ \nonumber
   \\ &\sum_p\log p\left(\frac{1}{p+1}\sum_{m=1}^\infty\frac{\alpha_p^{2m}+\overline{\alpha}_p^{2m}}{p^{m(1+2r)}}-\sum_{m=1}^\infty\frac{\alpha_p^{2m}+\overline{\alpha}_p^{2m}}{p^{m(1+2r)}}\right.\nonumber
   \\& \left.\,\,\,\,\,\,\,\,\,\,\,\,+\frac{\frac{\tau^*(p^2)}{p^{1+2r}}\frac{2\cdot\tau^*(p^2)}{p^{2+4r}}\left.\right.+\frac{3}{p^{3+6r}}}{1-\frac{\tau^*(p^2)}{p^{1+2r}}+\frac{\tau^*(p^2)}{p^{2+4r}}-\frac{1}{p^{3+6r}}}+\frac{1}{1-p^{1+2r}}\right)
  \end{align}
  
   as claimed.
 \end{proof}
 Evaluating this derivative at $r=iy$, we have
 \begin{align}
  B'(iy;iy) \ = \ &\sum_p\log p\left(\frac{1}{p+1}\sum_{m=1}^\infty\frac{\alpha_p^{2m}+\overline{\alpha}_p^{2m}}{p^{m(1+2iy)}}-\sum_{m=1}^\infty\frac{\alpha_p^{2m}+\overline{\alpha}_p^{2m}}{p^{m(1+2iy)}}\right.\nonumber
\\&\left.+\frac{\frac{\tau^*(p^2)}{p^{1+2iy}}-\frac{2\cdot\tau^*(p^2)}{p^{2+4iy}}+\frac{3}{p^{3+6iy}}}{1-\frac{\tau^*(p^2)}{p^{1+2iy}}+\frac{\tau^*(p^2)}{p^{2+4iy}}-\frac{1}{p^{3+6iy}}}+\frac{1}{1-p^{1+2iy}}\right)\label{bevaluated}.
 \end{align}
Equation \eqref{bevaluated} contains a term present in the number theory calculations, as well as an algebraically messy term after it.  The following vital lemma eliminates the extra term, greatly simplifying our expression and giving us perfect equality between $B'(iy;iy)$ and equation \ref{equaltob}.  Without it, our correspondence of terms between \ref{ntpred} and \ref{rcpred} would not work out as desired.

 \begin{lemma} We have
 \be-\sum_{m=1}^\infty\frac{\alpha_p^{2m}+\overline{\alpha}_p^{2m}}{p^{m(1+2iy)}}+
 \frac{\frac{\tau^*(p^2)}{p^{1+2iy}}-\frac{2\cdot\tau^*(p^2)}{p^{2+4iy}}+\frac{3}{p^{3+6iy}}}{1-
 \frac{\tau^*(p^2)}{p^{1+2iy}}+\frac{\tau^*(p^2)}{p^{2+4iy}}-\frac{1}{p^{3+6iy}}}+\frac{1}{1-p^{1+2iy}} \ = \ 0.
 \ee
 \end{lemma}

 \begin{proof}
 Set
 \be M\ = \ -\sum_{m=1}^\infty\frac{\alpha_p^{2m}+\overline{\alpha}_p^{2m}}{p^{m(1+2iy)}}+\frac{\frac{\tau^*(p^2)}{p^{1+2iy}}-\frac{2\cdot\tau^*(p^2)}{p^{2+4iy}}+\frac{3}{p^{3+6iy}}}{1-\frac{\tau^*(p^2)}{p^{1+2iy}}+\frac{\tau^*(p^2)}{p^{2+4iy}}-\frac{1}{p^{3+6iy}}}+\frac{1}{1-p^{1+2iy}}.
 \ee
 We rewrite the series that is the first term of  $M$.  Setting $A_1=\frac{\alpha_p^2}{p^{m(1+2iy)}}$ and $A_2=\frac{\overline{\alpha}_p^2}{p^{m(1+2iy)}}$ and noting that both have absolute value less than $1$, we have
 \begin{align}\sum_{m=1}^\infty\frac{\alpha_p^{2m}+\overline{\alpha}_p^{2m}}{p^{m(1+2iy)}} \ =\ &\sum_{m=1}^\infty\frac{\alpha_p^{2m}}{p^{m(1+2iy)}}+\sum_{m=1}^\infty\frac{\overline{\alpha}_p^{2m}}{p^{m(1+2iy)}}\nonumber
 \\ \ = \ &\sum_{m=1}^\infty A_1^m+\sum_{m=1}^\infty A_2^m\nonumber
 \\ \ = \ &\frac{A_1}{1-A_1}+\frac{A_1}{1-A_2}\nonumber
 \\ \ = \ &\frac{A_1+A_2-2A_1A_2}{1-A_1-A_2+A_1A_2}\nonumber
 \\ \ = \ &\frac{\frac{\alpha_p^2+\overline{\alpha}_p^2}{p^{1+2iy}}-\frac{2}{p^{2+4iy}}}{1-\frac{\alpha_p^2+\overline{\alpha}_p^2}{p^{1+2iy}}+\frac{1}{p^{2+4iy}}}.
 \end{align}  Using the fact that $\alpha_p^2+\overline{\alpha}_p^2=\tau^*(p^2)-1$, we have\begin{align}
 \sum_{m=1}^\infty\frac{\alpha_p^{2m}+\overline{\alpha}_p^{2m}}{p^{m(1+2iy)}} =\frac{\frac{\tau^*(p^2)-1}{p^{1+2iy}}-\frac{2}{p^{2+4iy}}}{1-\frac{\tau^*(p^2)-1}{p^{1+2iy}}+\frac{1}{p^{2+4iy}}}.
 \end{align}
Note that
 \begin{align}
 \frac{1}{1-p^{1+2iy}}=&-\frac{1/p^{1+2iy}}{1-(1/p^{1+2iy})}
 \end{align}  Letting $t$ denote $\tau^*(p)$ and $q$ denote $1/p^{1+2iy}$, we have

 \begin{align}M\ = \ &-\sum_{m=1}^\infty\frac{\alpha_p^{2m}+\overline{\alpha}_p^{2m}}{p^{m(1+2iy)}}+ \frac{\frac{\tau^*(p^2)}{p^{1+2iy}}-\frac{2\cdot\tau^*(p^2)}{p^{2+4iy}}+ \frac{3}{p^{3+6iy}}}{1-\frac{\tau^*(p^2)}{p^{1+2iy}}+\frac{\tau^*(p^2)}{p^{2+4iy}}- \frac{1}{p^{3+6iy}}}+\frac{1}{1-p^{1+2iy}}\nonumber
 \\ \ = \ &-\frac{\frac{\tau^*(p^2)-1}{p^{1+2iy}}-\frac{2}{p^{2+4iy}}}{1-\frac{\tau^*(p^2)-1}{p^{1+2iy}}+\frac{1}{p^{2+4iy}}}+\frac{\frac{\tau^*(p^2)}{p^{1+2iy}}-\frac{2\cdot\tau^*(p^2)}{p^{2+4iy}}+\frac{3}{p^{3+6iy}}}{1-\frac{\tau^*(p^2)}{p^{1+2iy}}+\frac{\tau^*(p^2)}{p^{2+4iy}}-\frac{1}{p^{3+6iy}}}-\frac{1/p^{1+2iy}}{1-(1/p^{1+2iy})}\nonumber
 \\ \ = \ &\frac{(1-t)q+2q^2}{1-(t-1)q+q^2}+\frac{tq-2tq^2+3q^3}{1-tq+tq^2-q^3}-\frac{q}{1-q}.
 \end{align}  Noting that $1-tq+tq^2-q^3=(1-(t-1)q+q^2)(1-q)$, we have
 \begin{align}
M=& \frac{(1-q)((1-t)q+2q^2)+tq-2tq^2+3q^3-(1-(t-1)q+q^2)q}{1-tq+tq^2-q^3}\nonumber
\\ \ = \ &\frac{q-qt+2q^2-q^2+tq^2-2q^3+tq-2tq^2+3q^3-q+tq^2-q^2-q^3}{1-tq+tq^2-q^3}\nonumber
\\ \ = \ &\frac{(1-t+t-1)q+(2-1+t-2t+t-1)q^2+(-2+3-1)q^3}{1-tq+tq^2-q^3}\nonumber
\\ \ = \ & 0.
 \end{align}
 \end{proof}
  Combined with equation \eqref{bevaluated}, this lemma immediately implies
 \begin{corollary}\label{prop}
 \be B'(iy;iy) \ = \ \sum_p\frac{\log p}{p+1}\sum_{m=1}^\infty\frac{\alpha_p^{2m}+\overline{\alpha}_p^{2m}}{p^{m(1+2iy)}}.\label{boosh}
 \ee
 \end{corollary}
 This gives us our second correspondence of terms.

 \subsection{The Error Term}\label{errorsection}

 From Proposition \ref{prop}, we have that equations \eqref{ntpred} and \eqref{rcpred} are in agreement save for the constant $\frac{g(0)}{2}$ and for the term
 \begin{align}
R(g;X)=-\frac{1}{LX^*}\int_{-\infty}^\infty g(\nu)&\sum_{d\leq X}     \left(\frac{d}{2\pi}\right)^{-2i\frac{\pi\nu}{L}}    \frac{\Gamma\left(6-i\frac{\pi\nu}{L}\right)}{\Gamma\left(6+i\frac{\pi\nu}{L}\right)}\nonumber
\\\times &\frac{\zeta\left(1+2i\frac{\pi\nu}{L}\right)L_\Delta\left({\rm sym^2},1-2i\frac{\pi\nu}{L}\right)}{L_\Delta({\rm sym^2},1)}B_\Delta\left(-i\frac{\pi\nu}{L},i\frac{\pi\nu}{L}\right)d\nu\label{betterbesmall}.
\end{align}  Our general technique will be to perform a contour shift and bound all the terms from $\Gamma$ onward in the expression above by a polynomial in $\nu$, then use the rapid decay of $g(\nu)$ to show the integral over $\nu$ converges so that we need only worry about $X$ terms (as well as a constant contribution from a pole that is balanced by the equal constant $g(0)/2$ in the explicit formula).

 First we will show that $ B_{\Delta}(-iy;iy)$ converges and will continue to do so for contour shifts of $y$ up to a cut-off point.

 \begin{proposition}\label{bconverges}  Let $-\frac{1}{2}+\gep' \le w \le \frac{1}{4} - \gep'$.  If we shift from $y$ to $y-iw$, then we have that $B_{\Delta}\left(-i(y-iw);i(y-iw)\right)$ is $O_w(1)$.
 \end{proposition}

 \begin{proof}
 We have
 \begin{align}
 B_{\Delta}(-iy;iy)=\prod_p&\left(1+\frac{p}{p+1}\left(\sum_{m=1}^\infty\frac{\tau^*(p^{2m})}{p^{m(1-2iy)}}-\frac{\tau^*(p)}{p}\sum_{m=0}^\infty\frac{\tau^*(p^{2m+1})}{p^{m(1-2iy)}}+\frac{1}{p^{1+2iy}}\sum_{m=0}^\infty\frac{\tau^*(p^{2m})}{p^{m(1-2iy)}} \right) \right)\nonumber
\\\times &\frac{\left(1-\frac{\tau^*(p^2)}{p^{1-2iy}}+\frac{\tau^*(p^2)}{p^{2-4iy}}-\frac{1}{p^{3-6iy}}\right)\left(1-\frac{1}{p^{1+2iy}}\right)}{\left(1-\frac{\tau^*(p^2)}{p}+\frac{\tau^*(p^2)}{p^2}-\frac{1}{p^{3}}\right)\left(1-\frac{1}{p}\right)}.
 \end{align}

 Letting $y'=y-iw$ (our shift), we will show that
\begin{align}\Theta(p)\ = \ &\left(1+\frac{p}{p+1}\left(\sum_{m=1}^\infty\frac{\tau^*(p^{2m})}{p^{m(1-2iy')}} - \frac{\tau^*(p)}{p}\sum_{m=0}^\infty\frac{\tau^*(p^{2m+1})}{p^{m(1-2iy')}} + \frac{1}{p^{1+2iy'}}\sum_{m=0}^\infty\frac{\tau^*(p^{2m})}{p^{m(1-2iy')}} \right) \right)\nonumber
\\ &\times \frac{\left(1-\frac{\tau^*(p^2)}{p^{1-2iy'}}+\frac{\tau^*(p^2)}{p^{2-4iy'}} - \frac{1}{p^{3-6iy'}}\right)\left(1-\frac{1}{p^{1+2iy'}}\right)}{\left(1-\frac{\tau^*(p^2)}{p}+\frac{\tau^*(p^2)}{p^2} - \frac{1}{p^{3}}\right)\left(1-\frac{1}{p}\right)}\nonumber\\ \ = \ & 1+O\left(\frac{1}{p^{2-4w-\epsilon}}\right)+O\left(\frac{1}{p^{2+2w-\epsilon}}\right),
\end{align} implying that the product converges as the error is $O(1/p^{1+\gep''})$. Noting that $\left|\frac{\tau^*(p^2)}{p}-\frac{\tau^*(p^2)}{p^2}+\frac{1}{p^{3}}\right|<1$ and $\left|\frac{1}{p}\right|<1$, we may rewrite  $\Theta(p)$ as
\begin{align}\Theta(p)=&\left(1+\frac{p}{p+1}\left(\sum_{m=1}^\infty\frac{\tau^*(p^{2m})}{p^{m(1-2iy')}}-\frac{\tau^*(p)}{p}\sum_{m=0}^\infty\frac{\tau^*(p^{2m+1})}{p^{m(1-2iy')}}+\frac{1}{p^{1+2iy'}}\sum_{m=0}^\infty\frac{\tau^*(p^{2m})}{p^{m(1-2iy')}} \right) \right)\nonumber
\\ &\times \ \left(1-\frac{\tau^*(p^2)}{p^{1-2iy'}}+\frac{\tau^*(p^2)}{p^{2-4iy'}} - \frac{1}{p^{3-6iy'}}\right)\left(1-\frac{1}{p^{1+2iy'}}\right)\nonumber\\ & \times \  \left(\sum_{n=0}^\infty\left(\frac{\tau^*(p^2)}{p} - \frac{\tau^*(p^2)}{p^2}+\frac{1}{p^{3}}\right)^n\right) \left(\sum_{n=0}^\infty\frac{1}{p^n}\right)
.\label{geomrewrite}
\end{align}

We now rewrite some of the infinite sums of \eqref{geomrewrite} as a main term plus an error term.  By Lemma \ref{sumlemma} in the appendix, for any $\epsilon>0$ we may truncate the terms of \eqref{geomrewrite} as follows, preserving multiple error terms depending on the direction in which $y$ has been shifted:
\begin{align}
\Theta{(p)}=&\left(1+\frac{p}{p+1}\left(\frac{\tau^*(p^2)}{p^{1-2iy'}}+\frac{\tau^*(p)^2}{p^{2}}+\frac{1}{p^{1+2iy'}}+
O\left(\frac1{p^{2-4w-\epsilon}}\right)+O\left(\frac{1}{p^{2-\epsilon}}\right) \right) \right)\nonumber
\\ &\times \left(1-\frac{\tau^*(p^2)}{p^{1-2iy'}}+O(\frac1{p^{2-4w}}) \right)\left(1-\frac{1}{p^{1+2iy'}}\right)
\nonumber\\ & \times \ \left(1+\frac{\tau^*(p^2)}{p} + O\left(\frac1{p^2}\right)\right)\left(1+\frac{1}{p}+O\left(\frac1{p^2}\right)\right)     \nonumber
\\ \ = \ &\left(1+\left(\frac{\tau^*(p^2)}{p^{1-2iy'}} - \frac{\tau^*(p^2)+1}{p} + \frac{1}{p^{1+2iy'}}\right)-\frac{1}{p+1}\left(\frac{\tau^*(p^2)}{p^{1-2iy'}} - \frac{\tau^*(p^2)+1}{p} + \frac{1}{p^{1+2iy'}}\right)\right.   \nonumber
\\ &\left. + O\left(\frac{1}{p^{2-4w-\epsilon}}\right)+ O\left(\frac{1}{p^{2-\epsilon}}\right)\right)\nonumber
\\ &\times\left(1-\frac{\tau^*(p^2)}{p^{1-2iy'}}- \frac{1}{p^{1+2iy'}}+\frac{\tau^*(p^2)+1}{p} + O\left(\frac1{p^{2-4w}}\right)\right)\nonumber
\\ \ = \ &\left(1+\frac{\tau^*(p^2)}{p^{1-2iy'}}-\frac{\tau^*(p^2)+1}{p} + \frac{1}{p^{1+2iy'}}+O\left(\frac{1}{p^{2-4w-\epsilon}}\right)+O\left(\frac{1}{p^{2+2w-\epsilon}}\right) + O\left(
\frac{1}{p^{2-\epsilon}}\right)\right)\nonumber
\\ & \times \left(1-\frac{\tau^*(p^2)}{p^{1-2iy'}} - \frac{1}{p^{1+2iy'}} + \frac{\tau^*(p^2)+1}{p} + O\left(\frac1{p^{2-4w}}\right)\right)\nonumber
\\ \ = \ &1-\frac{\tau^*(p^2)}{p^{1-2iy'}}-\frac{1}{p^{1+2iy'}} + \frac{\tau^*(p^2)+1}{p}+\frac{\tau^*(p^2)}{p^{1-2iy'}}- \frac{\tau^*(p^2)+1}{p}+\frac{1}{p^{1+2iy'}}\nonumber
\\& + O\left(\frac{1}{p^{2-4w-\epsilon}}\right)+O\left(\frac{1}{p^{2+2w-\epsilon}}\right)\nonumber
\\ \ = \ &1+O\left(\frac{1}{p^{2-4w-\epsilon}}\right)+O\left(\frac{1}{p^{2+2w-\epsilon}}\right).
\end{align}
If we have shifted in the positive direction, the larger error term is $O\left({1}/{p^{2-4w-\epsilon}}\right)$, meaning we may shift $2\epsilon$ close to $1/4$ and have $O(1/p^{1+\delta})$ where $\delta>0$.  Similarly, if we have shifted in the negative direction, the larger error term is $O\left({1}/{p^{2+2w-\epsilon}}\right)$, meaning we may shift $2\epsilon$ close to $-1/2$ and have $O(p^{1+\delta})$ where $\delta>0$.  In both cases we have convergence with the product $O_w(1)$.
 \end{proof}

 Thus the $B_\Delta$ term is $O_w(1)$ for fixed $w$ and any $y$ (as the bound is independent of the imaginary part of our variable). To bound $\zeta$ and $L_\Delta$, we use the standard fact (see for example \cite{IK}) that both grow polynomially in $y$ as $|y|\rightarrow\infty$ (where $y$ is our imaginary part).

We now consider $\frac{\Gamma(6-iy)}{\Gamma(6+iy)}$. Our shifting restrictions from $B_\Delta$ allow us to only consider $w$ with $-1/2 < w < 1/4$, and thus we will not shift far enough to reach a pole or zero for either the numerator or the denominator. If we shift from $y$ to $y-iw$, this term becomes
\be
\frac{\Gamma(6-i(y-iw))}{\Gamma(6+i(y-iw))} \ = \ \frac{\Gamma(6-w-iy)}{\Gamma(6+w+iy)}.
\ee
From the definition of the Beta function we know
\be
\frac{\Gamma(x)\Gamma(z)}{\Gamma(x+z)} \ = \ \int_0^1t^{x-1}(1-t)^{z-1}dt.
\ee
Taking $x=6-w-iy$ and $z=2(w+iy)$, we have
\be
\frac{\Gamma(6-w-iy)\Gamma(2(w+iy))}{\Gamma(6+w+iy)} \ = \ \int_0^1t^{5-w-iy}(1-t)^{2(w+iy)-1}dt \ = \ O_{w,y}(1).
\ee
Thus
\be\frac{\Gamma(6-w-iy)}{\Gamma(6+w+iy)} \ = \ \frac{O_{v,w}(1)}{\Gamma(-2(w+iy))},
\ee giving us a polynomial bound on $\frac{\Gamma(6-w-iy)}{\Gamma(6+w+iy)}$ in $y$ due to the properties of the $\Gamma$ function.

The last detail to attend to before attacking \eqref{betterbesmall} is to bound $g\left(\nu-i\frac{wL}{\pi}\right)$.
\begin{proposition}  For any $w$, we have $g\left(\nu-iw\frac{\log(X/2\pi)}{\pi}\right)\ll X^{2\sigma |w|}\left(\nu^2+(w\frac{\log(X/2\pi)}{2\pi})^2\right)^{-B}$ for any $B\geq 0$.\label{boundedsch}
\end{proposition}

\begin{proof}
Since $g(\nu)=\int_{\omega = -\infty}^\infty\hat{g}(\omega)e^{2\pi i\omega\nu}d\omega$, integrating by parts $2n$ times (and noting the boundary terms vanish as $\widehat{g}$ is supported in $(-\sigma, \sigma)$), we find
\begin{align}
g(\nu-iy)& \ =\ \int_{\omega = -\infty}^\infty\hat{g}(\omega)e^{2\pi i\omega(\nu-iy)}d\omega\nonumber
\\& \ = \ \int_{\omega = -\infty}^\infty\hat{g}^{(2n)}(\omega)(2\pi i(\nu-iy))^{-2n}e^{2\pi i (\nu-iy)\omega}d\omega\nonumber
\\& \ \le \ \int_{\omega = -\infty}^\infty\left|\hat{g}^{(2n)}(\omega)\right|  \cdot \left|2\pi i(\nu-iy)\right|^{-2n}\cdot e^{2\pi |y \omega|}d\omega\nonumber
\\ &\ \ll\ e^{2\pi |y|\sigma}|\nu-iy|^{-2n}
\end{align} for any $n\in\mathbb{Z}^+$.  Our claim follows by taking $y=w\frac{\log(X/2\pi)}{\pi}$.
\end{proof}
We are now ready to analyze \eqref{betterbesmall}, which we may rewrite as
 \begin{align}
R(g;X)\ = \ &-\frac{1}{LX^*}\int_{-\infty}^\infty g(\nu)\sum_{d\leq X}     \left(\frac{d}{2\pi}\right)^{-2i\frac{\pi\nu}{L}}    \frac{\Gamma\left(6-i\frac{\pi\nu}{L}\right)}{\Gamma\left(6+i\frac{\pi\nu}{L}\right)}\nonumber
\\&\times \frac{\zeta\left(1+2i\frac{\pi\nu}{L}\right)L_\Delta\left({\rm sym^2},1-2i\frac{\pi\nu}{L}\right)}{L_\Delta({\rm sym^2},1)}B_\Delta\left(-i\frac{\pi\nu}{L},i\frac{\pi\nu}{L}\right)d\nu\nonumber
\\ \ = \ &-\frac{1}{LX^*\delta}\int_{-\infty}^\infty g(\nu)  \sum_{d\leq X}  e^{-2i\frac{\pi\nu}{L}\log{\frac{d}{2\pi}}}    \frac{\Gamma\left(6-i\frac{\pi\nu}{L}\right)}{\Gamma\left(6+i\frac{\pi\nu}{L}\right)}\nonumber
\\&\times{\zeta\left(1+2i\frac{\pi\nu}{L}\right)L_\Delta\left({\rm sym^2},1-2i\frac{\pi\nu}{L}\right)}B_\Delta\left(-i\frac{\pi\nu}{L},i\frac{\pi\nu}{L}\right)d\nu\label{errorrewritten}
\end{align} where $\delta=L_\Delta({\rm sym^2},1)$, a nonzero constant.  By Proposition \ref{bconverges}, we may shift $y$ to $y-iw$ where $-\frac{1}{2}<w<\frac{1}{4}$ and still have $B_\Delta(-iy,iy)$ converge.  Scaling by $L/\pi$, this becomes a shift to $\nu-i\frac{wL}{\pi}$.  Replacing $\nu$ with $\nu-i\frac{wL}{\pi}$ where $w=0$ (we will shift momentarily), equation \eqref{errorrewritten} becomes
\begin{align}
R(g;X)=&-\frac{1}{LX^*\delta}\int_{-\infty}^\infty g\left(\nu-i\frac{wL}{\pi}\right)
    \sum_{d\leq X} e^{-2\pi i\left(\nu-i\frac{wL}{\pi}\right)\frac{\log{({d}/{2\pi})}}{L}}
        \frac{\Gamma\left(6-w-i\frac{\pi\nu}{L}\right)}
        {\Gamma\left(6+w+i\frac{\pi\nu}{L}\right)}\nonumber
\\&     \times           {\zeta\left(1+2w+2i\frac{\pi\nu}{L}\right)
     L_\Delta\left({\rm sym^2},1-2w-2i\frac{\pi\nu}{L}\right)}
B_\Delta\left(-i\frac{\pi\nu}{L}-w,i\frac{\pi\nu}{L}+w\right)d\nu.\label{zeroshift}
\end{align}
Shifting the contour to $w=\frac{1}{4}-\epsilon$, the only residue contribution due to our shift comes from the pole of $\zeta\left(1+2w+2i\frac{\pi\nu}{L}\right)$ at $w=\nu=0$. At $w=\nu=0$ we have that the $d$-sum is simply $X^*$.  The pole of $\zeta(s)$ is simply $1/(s-1)$, and since $s=1+2i\frac{\pi\nu}{L}$ the $\frac{1}{\nu}$ term from the zeta function has coefficient $\frac{L}{2\pi i}$.  Applying the residue theorem after replacing the integral from $-\epsilon$ to $\epsilon$ with a semi-circle oriented clockwise, we lose a factor of $\frac{1}{2\pi i}$, giving us $L$ which is cancelled by the outside coefficient.  The contribution of the pole is therefore $1/2$ everything else evaluated at $\nu=0$, yielding a contribution of $g(0)/2$ which perfectly cancels the constant term in the number theory analysis.
  We also have by Lemma \ref{secondmillerlemma}  that
\be  \sum_{d\leq X} e^{-2\pi i\left(\nu-i\frac{wL}{\pi}\right)\frac{\log{({d}/{2\pi})}}{L}}
\ = \ X^*e^{-2\pi i \left(\nu-i\frac{wL}{\pi}\right)}\left(1-\frac{2\pi i \left(\nu-i\frac{wL}{\pi}\right)}{L}\right)^{-1}+O(X^{2\epsilon}).
\ee

With $w=\frac{1}{4}-\epsilon$, this gives us that the sum is $O(X^*X^{-1/2+\epsilon})$ (possibly modifying our $\epsilon$, but still keeping it arbitrarily small).  As this term is independent of $\nu$, it will not harm the convergence of our integral.  By our previous arguments, we have
\begin{align}
&\frac{\Gamma\left(6-w-i\frac{\pi\nu}{L}\right)}{\Gamma\left(6+w+i\frac{\pi\nu}{L}\right)}{\zeta\left(1+2w+2i\frac{\pi\nu}{L}\right)}\nonumber
\\&\times {L_\Delta\left({\rm sym^2},1-2w-2i\frac{\pi\nu}{L}\right)}B_\Delta\left(-i\frac{\pi\nu}{L}-w,i\frac{\pi\nu}{L}+w\right)\nonumber
\\=&O(\nu^k)
\end{align}
for some positive integer $k$.  Choosing $n=k+1$ in Proposition \ref{boundedsch} now yields \be g\left(\nu-iw\frac{\log(X/2\pi)}{\pi}\right)\ \ll\  X^{2\sigma |w|}\left(\nu^2+\left(w\frac{\log(X/2\pi)}{2\pi}\right)^2\right)^{-(k+1)},\ee allowing us to replace the $g$-term in the integrand in \eqref{zeroshift} with $\nu^{-2k-2}$. This cancels the polynomial size of the other terms (which is $O(v^k)$), giving us convergence of the integral in $\nu$.  As we are only interested in the order of magnitude in terms of $X$ (and not the constant that arises from integrating with respect to $\nu$), it is enough to note that
\bea
R(g;X) & \ = \ &O\left(\frac{1}{X^*}X^{2\sigma |w|}X^*X^{-1/2+\epsilon}\right)\nonumber\\
& \ = \ &O\left(X^{-\frac12+2\sigma |w|+\epsilon}\right)\nonumber\\
& \ = \ &O\left(X^{-(1-\sigma)/2+\epsilon}\right)
\eea as $|w| < 1/4$ (note that $\epsilon$ may change value, but is still arbitrarily small). Thus $R(g,X)$ contributes an error erm $O(X^{-(1-\sigma)/2+\epsilon})$, into which we may absorb the error term of Equation \ref{rcpred}.  Having either matched all other terms or shown they are sufficiently small, we have proven our main result, namely Theorem \ref{thm:mainnumbequalsratios}, which asserted that the Ratios Conjecture's prediction is correct to the stated accuracy. In other words, in this instance of convolving an infinite family with a family of size $1$, we have verified that the Ratios Conjecture's prediction is correct up to a power savings error for suitably restricted test functions (requiring $\sigma<1$).

\subsection{Where To Go From Here}

Ideally, our result can be improved to decrease the error term in the Ratios Conjecture prediction.  The limiting factor in obtaining this bound is the limitations on our contour shift imposed by the $B_\Delta$ term.  If this term were more well understood, it might be possible to extend our shift further to make the error term \ref{betterbesmall} even smaller. Note that in \cite{Mil3}, where the family under consideration is just Dirichlet $L$-functions arising from even fundamental discriminants, the corresponding term can be rewritten from the product expansion in \cite{CS1} and identified as a product of zeta functions. This allows a much better analysis of the error term; to date we have not been able to determine a similar simplification here.

In terms of applying these techniques to other $L$-functions, a natural family is the quadratic twists of an elliptic curve $L$-functions.  The number theory analysis of these families is very similar to the analysis of the twists of the tau $L$-function, with the primary difference in the $S_{\rm even}$ terms due to a special prime (the conductor of the elliptic curve).  With our understanding of the analysis for the twisted tau family, it will easier to narrow in on trouble spots in the analyses of these other families.

\newpage

\section{Random Matrix Theory Models of Convolutions}\label{section4}

We conclude by discussing the possibility of modeling convolutions more generally using Random Matrix Theory.  To do so, we will test whether methods of combining collections of matrices exhibit the same symmetries and relations found in convolving families of $L$-functions.

We begin by defining the symmetry constant for a family $\mathcal{F}$ of $L$-functions.

\begin{definition}  The symmetry constant of a family $\mathcal{F}$, denoted $c_{\mathcal{F}}$, is defined to equal $0$ (respectively $1$ or $-1$) if the one-level density of the family agrees with unitary (respectively symplectic or orthogonal) matrices.
\end{definition}

In \cite{DM2}, Miller and Due\~{n}ez prove that for ``nice'' families $\mathcal{F}$ and $\mathcal{G}$ (what they call NT-good families of unitary automorphic representations of $GL_n(\mathbb{A}_{\mathbb{Q}})$ and $GL_m(\mathbb{A}_{\mathbb{Q}})$ with trivial central character) with symmetry constants $c_\mathcal{F}$ and $c_\mathcal{G}$ such that the Rankin-Selberg convolution $\mathcal{F}\times\mathcal{G}$ is an NT-good family, we have that $\mathcal{F}\times\mathcal{G}$ has symmetry constant
$$c_ {\mathcal{F}\times\mathcal{G}}=c_\mathcal{F}\cdot c_\mathcal{G}.$$
This raises the question of what the corresponding operation is in Random Matrix Theory; that is, what method of combining matrices gives eigenvalue statistics that can be predicted based on the groups of matrices being combined.  A natural candidate for this operation is the tensor product (or Kronecker product) of matrices.  We use Mathematica to compute eigenvalue statistics for tensor products of unitary, unitary orthogonal, and unitary symplectic matrices.

\subsection{Types of Matrices}

The following definitions and facts come from \cite{matrix}.

\begin{definition} An $N\times N$ matrix $X$ with complex entries is said to be \emph{unitary} if $XX^*=I$ (where $X^*$ denotes the conjugate transpose of $X$); we denote the group of all $N\times N$ unitary matrices by $U(N)$.

A unitary matrix $X$ is said to be \emph{orthogonal} if $XX^T=I$, where $X^T$ denotes the transpose of $X$; we denote the group of all $2N\times 2N$ orthogonal matrices by $SO(2N)$ and of all $(2N+1)\times (2N+1)$ orthogonal matrices by $SO(2N+1)$.

A unitary matrix $X$ is said to be \emph{symplectic} if $XZX^T=Z$ (where $Z=\left(\begin{smallmatrix} 0&I_N\\ -I_N&0 \end{smallmatrix}\right)$ with $I_N$ the $N\times N$ identity matrix); we denote the group of $2N\times 2N$ symplectic matrices by $USp(2N)$.
\end{definition}

All eigenvalues of unitary matrices have absolute value $1$, and so must be of the form $e^{i\theta}$ for some $0\leq\theta<2\pi$.  The eigenvalues of $X\in U(N)$ are

$$e^{i\theta_1}, e^{i\theta_2}, \dots, e^{i\theta_N} $$
where
$$0 \leq \theta_1 \leq \theta_2 \leq \dots \leq \theta_N < 2 \pi.$$
These eigenangles are distributed uniformly on $[0,2\pi)$.

For any eigenvalue of an orthogonal or symplectic matrix, its complex conjugate is also an eigenvalue.  For $X\in SO(2N)$, we have eigenvalues
$$e^{\pm i\theta_1}, e^{\pm i\theta_2}, \dots, e^{\pm i\theta_N} $$
where
$$0 \leq \theta_1 \leq \theta_2 \leq \dots \leq \theta_N \leq \pi.$$
These eigenangles are distributed according to the probability distribution
$$\frac{2^{(N-1)^2}}{\pi^NN!}\prod_{1\leq j<k\leq M}(\cos\theta_k-\cos\theta_j)^2. $$

For $X\in SO(2N+1)$, we have eigenvalues
$$1,e^{\pm i\theta_1}, e^{\pm i\theta_2}, \dots, e^{\pm i\theta_N} $$
where
$$0 \leq \theta_1 \leq \theta_2\ \leq \dots \leq \theta_N \leq \pi.$$
These eigenangles are distributed according to the probability distribution
$$\frac{2^{N^2}}{\pi^NN!}\prod_{1\leq j<k\leq M}(\cos\theta_k-\cos\theta_j)^2\prod_{h=1}^N\sin^2\frac{\theta_h}{2}. $$

For $X\in USp(2N)$, we have eigenvalues
$$e^{\pm i\theta_1}, e^{\pm i\theta_2}, \dots, e^{\pm i\theta_N} $$
where
$$0 \leq \theta_1 \leq \theta_2 \leq \dots \leq \theta_N \leq \pi.$$
These eigenangles are distributed according to the probability distribution
$$\frac{2^{(N-1)^2}}{\pi^NN!}\prod_{1\leq j<k\leq M}(\cos\theta_k-\cos\theta_j)^2\prod_{h=1}^N\sin^2\theta_h. $$

These are the groups of matrices on which we will be performing our numerical investigations.  The quantities of interest are the eigenvalues of the tensor products of these matrices.  Since the eigenvalues of a tensor product are simply all possible products of an eigenvalue from the first matrix with an eigenvalue from the second matrix, we have that the eigenangles simply add under tensor product.  It will therefore suffice to generate eigenangles from each grou's eigenangle distribution and add them.  The statistic we focus on is the lowest eigenangle statistic (modulo $2\pi$).

\subsection{Sampling from Distributions}

Sampling from unitary matrices is simple, as the eigenangles are uniformly distributed on $[0,2\pi)$. To randomly sample from the orthogonal and symplectic distributions, we utilize the Accept-Reject method (from the notes of Carston Botts):

  To generate a random variable $X$ that is distributed according to $f_X(x)$, find another density $g(x)$ such that
$$S=\sup_x\left[\frac{f_X(x)}{g(x)}\right]<\infty.$$
Set $M\geq S$, and perform the following algorithm:

\begin{itemize}
\item[1.]  Generate a candidate value of $X$, which we denote by $X^c$, from $g(x)$.
\item[2.] Generate $U$ distributed according to ${\rm Unif}(0,Mg(X^c))$
\item[3.] If $U\leq f_X(X^c)$, accept $X^c$ as a draw from $f_X(x)$.
\end{itemize}

Note that for orthogonal and symplectic matrices, we have $f_X(x)$ is nonzero on $(0,\pi)^N$ and $0$ elsewhere.  We may therefore utilize $g(x)={\rm Unif}((0,\pi)^N)$, allowing us to use $M=S={\pi^N}\sup_x[f_X(x)]$.  Thus we may generate $U$ according to $$\text{Unif}(0,Mg(X^c))=\text{Unif}\left(0,{\pi^N}\sup_x[f_X(x)]\cdot\frac{1}{\pi^N}\right)=\text{Unif}(0,\sup_x[f_X(x)]).$$

Numerically estimating $\sup_x[f_X(x)]$ and increasing it by a safe margin, we will use the above algorithm to generate data for orthogonal and symplectic matrices.

\subsection{Combinations of Orthogonal and Symplectic Matrices}

From the results of Due\~{n}ez and Miller, we expect orthogonal combined with orthogonal to yield symplectic (since $-1\cdot -1=1$), and for symplectic combined with symplectic to yield symplectic.  To test these hypotheses, we have used Mathematica to generate data on the lowest eigenangle statistics of $N^2\times N^2$ orthogonal and symplectic matrices and of tensor products of two $N\times N$ orthogonal matrices, two $N\times N$ symplectic matrices, and one of each. For convenience, we will let ``orth'' (resp. ``symp'') refer to $N^2\times N^2$ orthogonal (resp. symplectic) matrices, and ``orth/orth'', ``symp/symp'', and ``orth/symp'' refer to tensor products of the corresponding $N\times N$ matrices.

There was little similarity found in comparing the distributions corresponding to $N^2\times N^2$ matrices with the tensor products of smaller matrices, as illustrated in \ref{fourplots}.  The differences between the compared matrices include noticeable qualitative differences (such as symp's repulsion from $0$), as well as subtler cases involving scale (the tensor products seem to have tails that spread further).  Although we have illustrated only the $N=4$ case, qualitative attributes persist in cases of higher $N$ (for instance, the repulsion of symp), implying that these differences exist in the limit.

There is a notable similarity between the distributions for orth/orth and symp/symp, as illustrated in Figures \ref{4pdf} through \ref{12pdf}.  Although there seems to be a systematic bias that causes symp/symp to overshoot orth/orth for a time, and then reverse, this phenomenon seems to diminish as $N$ increases, implying that in the limit these two distributions are equal.  This implies that perhaps the model is salvagable, as the combination of orthogonal with orthogonal should look like the combination of symplectic with symplectic (since $-1\cdot -1=1\cdot 1$).  To test this, we have looked at orth/orth/orth and symp/symp/symp, which would be different if this multiplicative structure were different.  However, these histograms are very similar as well.  Although an interesting phenomenon seems to be causing these distributions to approach one another, it seems fundamentally different from the multiplication of symmetry constants.

\begin{figure}[hbt]
\begin{center}
\includegraphics[
height=2.50in
]
{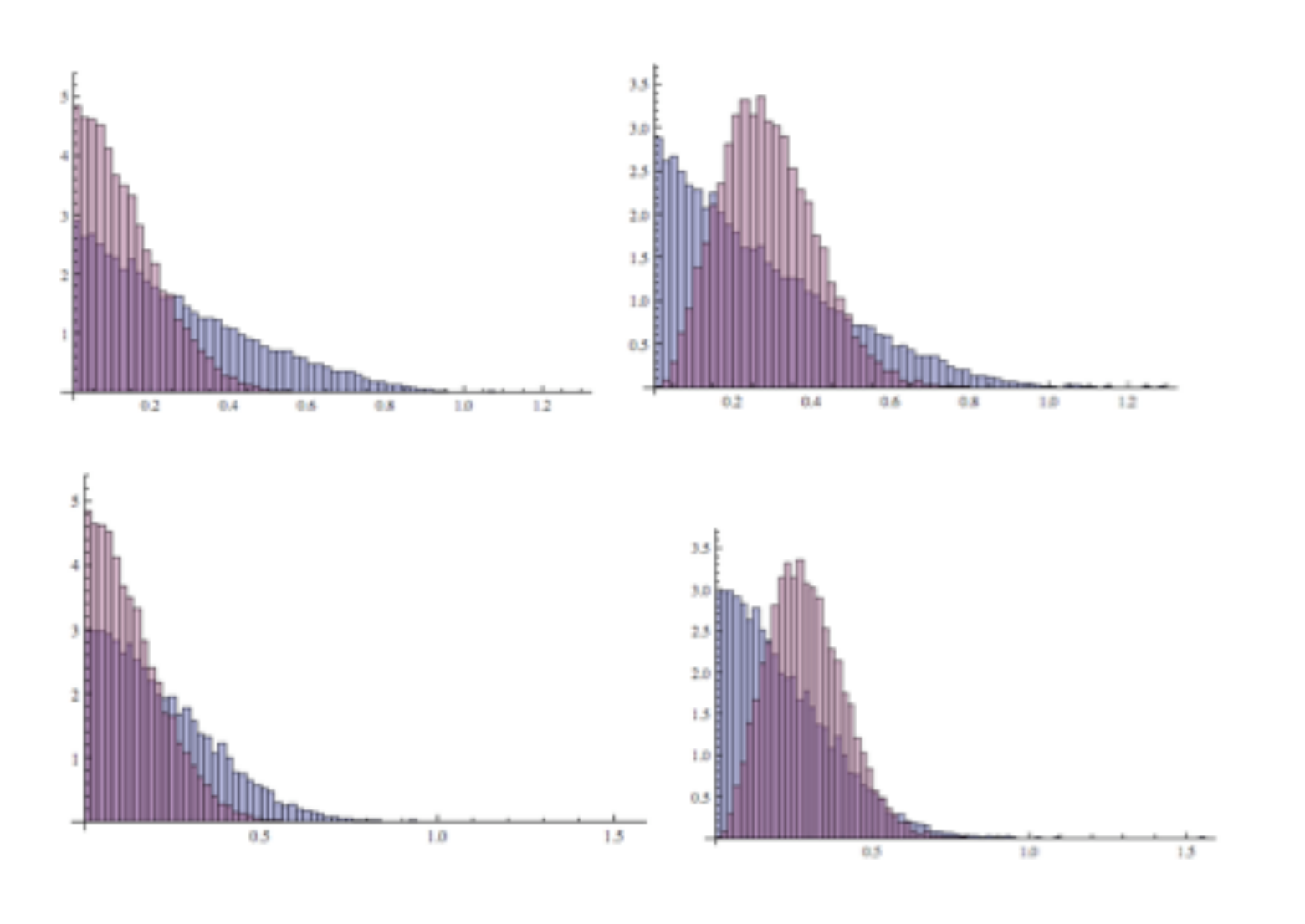}
\caption{Lowest eigenangle histograms in the case of $N=4$.  From left to right then top down, we have the following comparisons: orth/orth vs. orth, orth/orth vs. symp, symp/symp vs. orth, and symp/symp vs. orth. (Blue is the first type, red is the second type.)}
\label{fourplots}
\end{center}
\end{figure}

\begin{figure}[hbt]
\begin{center}
\includegraphics[
height=1.5in
]
{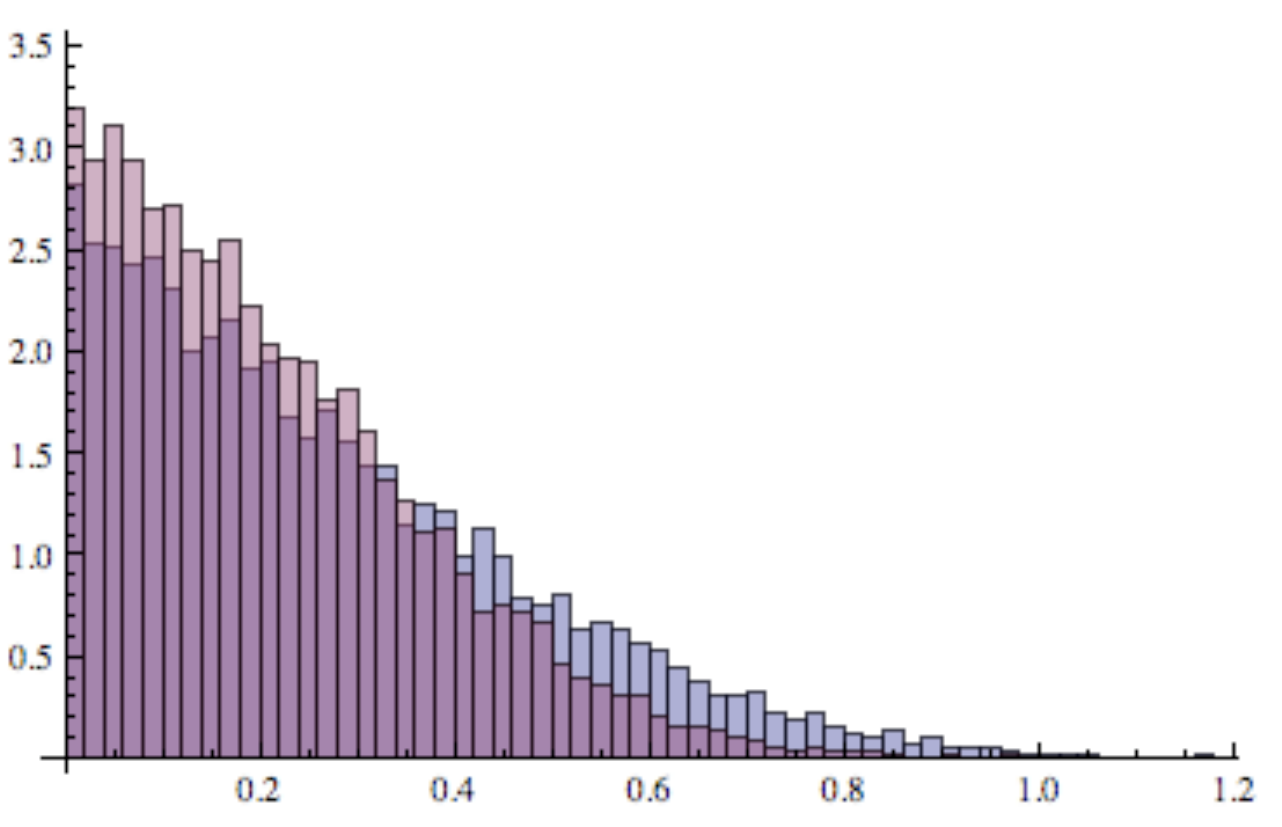}
\caption{Histogram of smallest eigenangles for orth/orth (blue) vs. symp/symp (red) when $N=4$.}
\label{4pdf}
\end{center}
\end{figure}

\begin{figure}[hbt]
\begin{center}
\includegraphics[
height=1.5in
]
{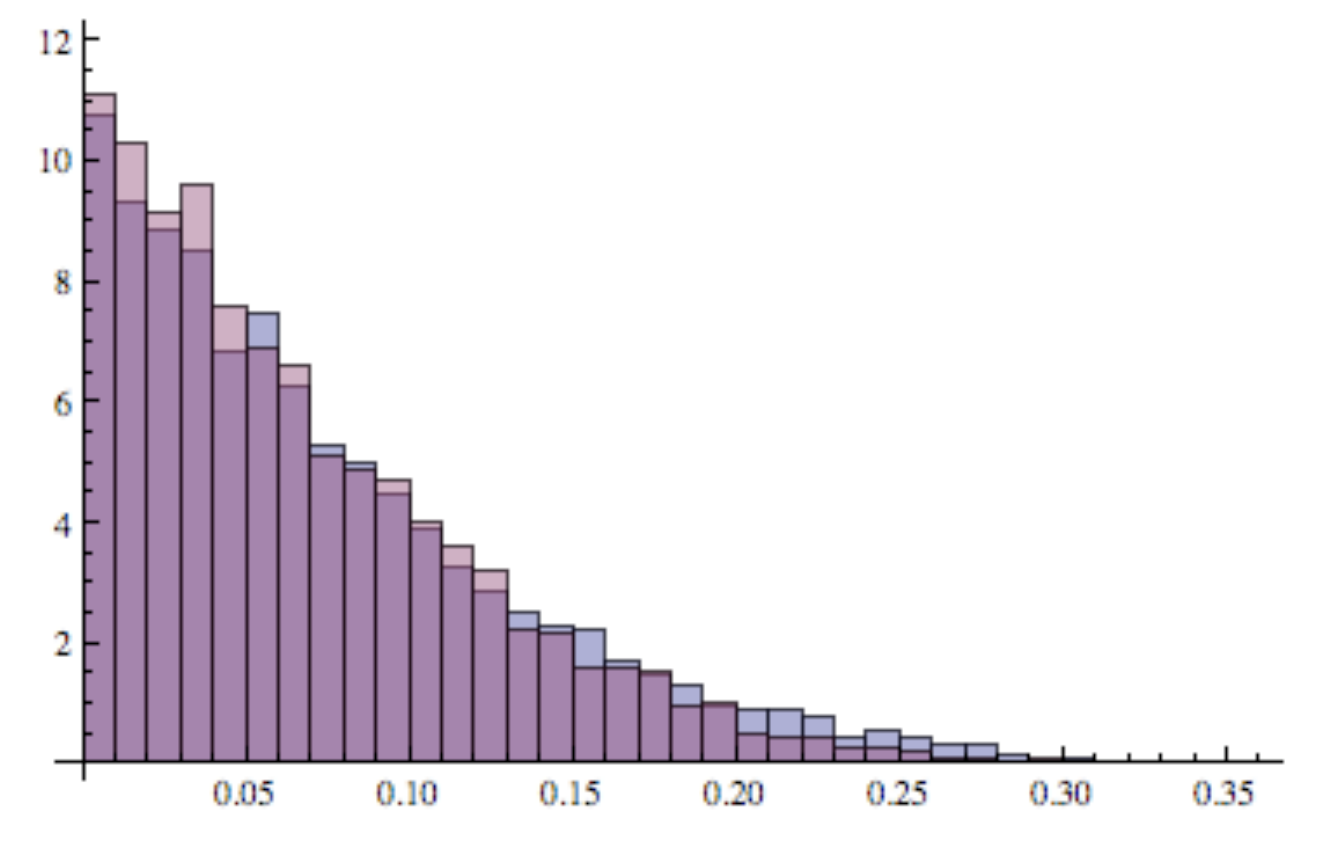}
\caption{Same as Figure \ref{4pdf} with $N=8$.}
\label{8pdf}
\end{center}
\end{figure}

\begin{figure}[hbt]
\begin{center}
\includegraphics[
height=1.5in
]
{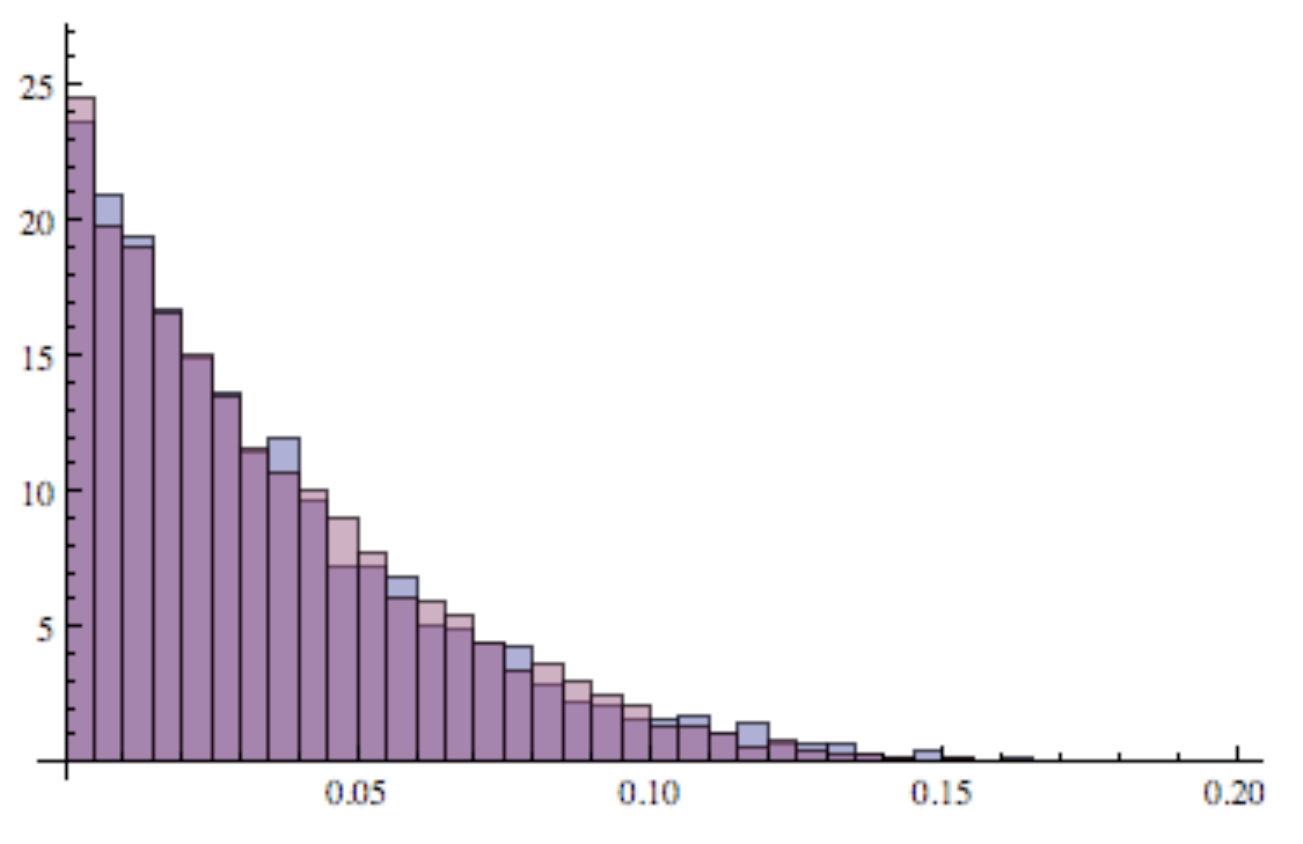}
\caption{Same as Figure \ref{4pdf} with $N=12$.}
\label{12pdf}
\end{center}
\end{figure}

\subsection{What Happens When We Include Unitary}

Although the predicted eigenangle statistics fail to hold for combinations of orthogonal and symplectic matrices, it seems reasonable that the ``$0$'' role of unitary will hold.  As its eigenangles are uniformly distributed, it is not unreasonable to conjecture that a combination of unitary with any other type of matrix will yield unitary.

If we look at the eignangle statistics for unitary/orth and unitary/symp, we see great similarity; indeed, considering Figure \ref{uorthusymp}, any difference between the two distributions seems to be random noise.  However, comparing these distributions with unitary/unitary shows a significant bias in overshooting/undershooting, suggesting that the structure is not as simple as multiplication by $0$; see Figure \ref{unitaryunitary}

\begin{figure}[hbt]
\begin{center}
\includegraphics[
height=1.5in
]
{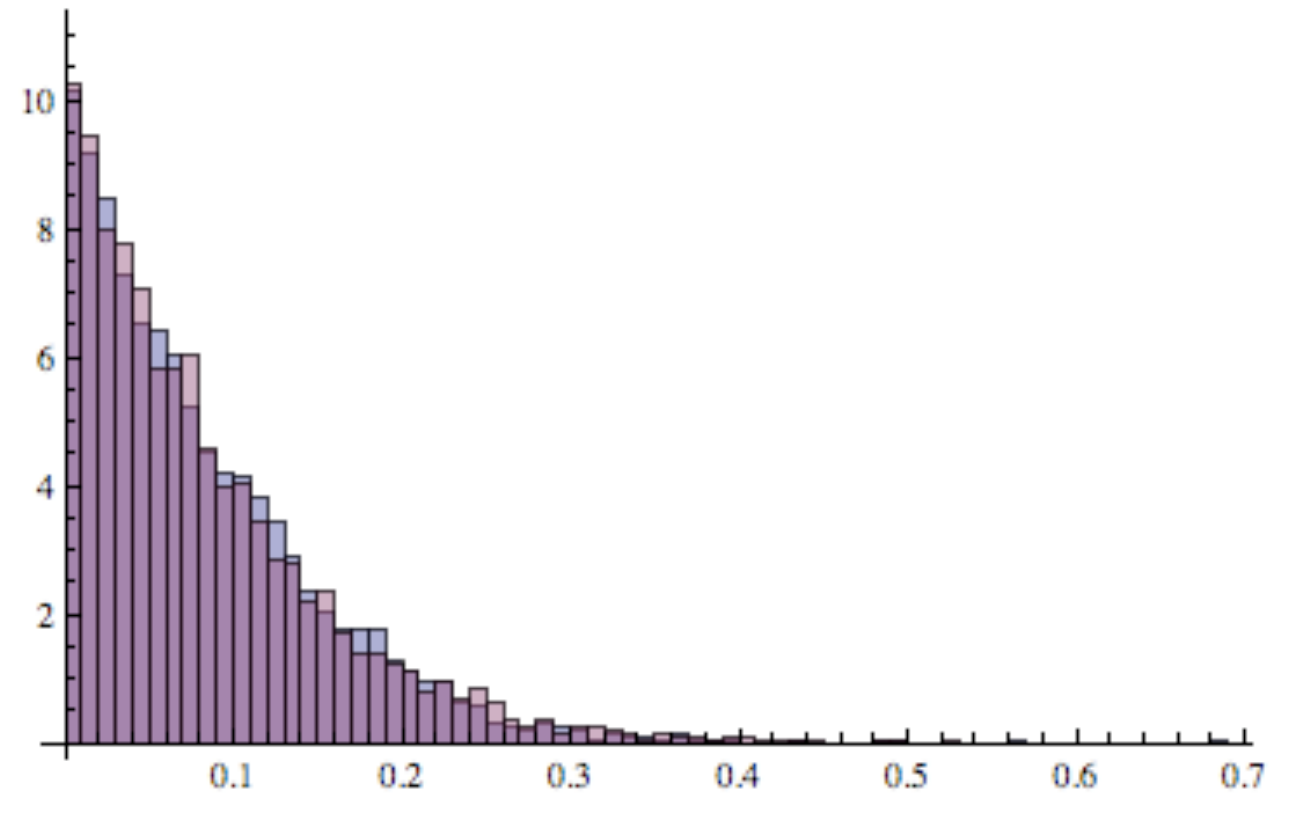}
\caption{Lowest eigenangle statistics for unitary/orth and unitary/symp for $N=8$.}
\label{uorthusymp}
\end{center}
\end{figure}

\begin{figure}[hbt]
\begin{center}
\includegraphics[
height=1.5in
]
{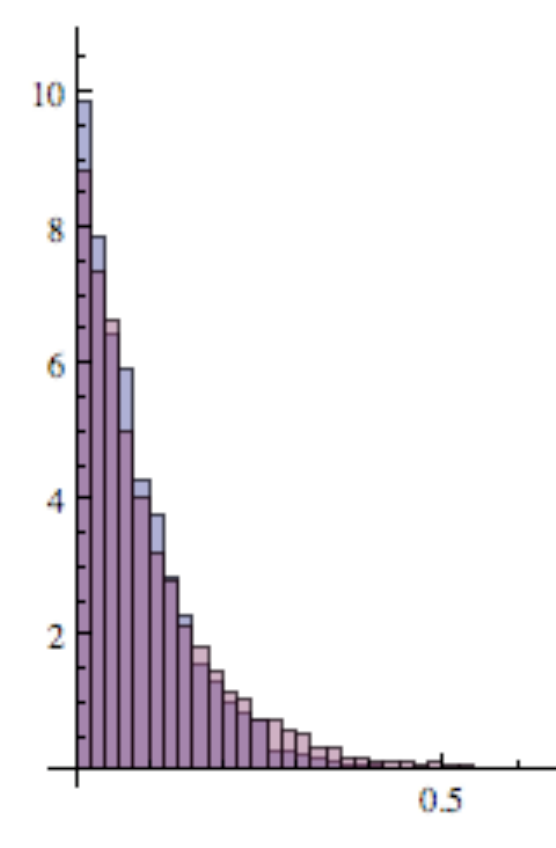}
\caption{Lowest eigenangle statistics for unitary/symp and unitary/unitary for $N=8$.}
\label{unitaryunitary}
\end{center}
\end{figure}

\subsection{Other Methods of Combining Matrices}

There are other methods of combining two matrices that could potentially model the Rankin-Selberg convolution.  These include the Tracy-Singh and Khatri-Rao products, which act on partitioned matrices.  A natural partition of a $2N\times 2N$ matrix being into four equal-sized parts, we have investigated the behavior of matrices under these operations.  In the case of the Tracy-Singh product, we find that eigenvalues are multiplicative as they are under the tensor product.  It follows that this product will affect eigenangle statistics in precisely the same fashion as the tensor product, giving us no new information.  The Khatri-Rao product behaves differently.  In the case of combining two diagonal matrices, the resulting eigenvalues are a (usually proper) subset of the products of eigenvalues of the original matrices; this implies that the ordering of eigenvalues along the diagonal of a matrix matters, something that is not usually considered when looking at entire families of matrices.  In addition, when considering non-diagonal unitary matrices, the resulting eigenvalues do not even necessarily have magnitude $1$, rendering eigenangles an unnatural statistic to study.

\subsection{Random Matrix Conclusions}

Based on both qualitative and quantitative results, it seems that taking the Kronecker product of matrices (as well as applying Tracy-Singh or Khatri-Rao) is an inadequate model for convolving families of $L$-functions.  However, these findings suggest certain interesting patterns in eigenangle statistics of Kronecker products of different families, such as $\text{orth}^n$ and $\text{symp}^n$ appearing to approach one another.  While not necessarily useful to modeling $L$-functions, rigorously exploring these patterns may be interesting in their own right.

\newpage

\appendix

\section{Key Lemmas}\label{sec:appendixkeylemmas}

In this appendix we include three key lemmas used in sections \ref{section2} and \ref{section3}.  For the first three lemmas, two on sums over fundamental discriminants and one of $S_{even,1}$, the lemmas and the proofs thereof are taken or modified (with permission) from \cite{Mil3}.  For the third, we write sums from Proposition \ref{bconverges} as a main term plus an error term.

%\begin{lemma}\label{sumofone}
%\be
%\sum_{\substack{d \leq X \\ p|d}}1 =\frac{X^*}{p-1}+O(X^{1/2})\label{sumofone2}.
%\ee
%\end{lemma}
%(Note:  this lemma and the proof thereof come from \ref{quadtwist}.)
%\begin{proof} COPY AND PASTE PROOF
%\end{proof}

\begin{lemma}\label{firstmillerlemma} Let $d$ denote an even
fundamental discriminant at most $X$, and set $X^\ast = \sum_{d \le
X} 1$. Then \be X^\ast   \ = \ \frac{3}{\pi^2}X + O(X^{1/2}) \ee and
for $p \le X^{1/2}$ we have \be \sum_{d \le X \atop p|d} 1 \ = \
\frac{X^\ast}{p+1} + O(X^{1/2}). \ee
%Note that $X \ll X^\ast
%\ll X$. If $p=2$ then the above is true only if our family is all
%fundamental discriminants; if it is just the even fundamental
%discriminants then the answer is zero (even fundamental
%discriminants are congruent to 1 modulo 4, and thus none of them are
%divisible by 2).
\end{lemma}

\begin{proof} We first prove the claim for $X^\ast$, and then
indicate how to modify the proof when $p|d$.
%We must show, for each $p \le X^{1/2}$, that
%\be\label{eq:sumpdividingd} \sum_{d \le X \atop p|d}
% 1 \ = \  \frac{X^\ast}{p+1} + O(X^{1/2}). \ee
%\footnote{If $F(s) = \sum_n
%a_n n^{-s}$ is a Dirichlet series with non-negative coefficients
%convergent for $\Re(s) > 1$, analytic for $\Re(s) \ge 1$ except for
%$s=1$, and $F(s) = H(s)/ (s-1)$ for a holomorphic $H(s)$ for $\Re(s)
%\ge 1$, then $\sum_{n\le x} \sim H(1)x$.}.

Let $\mu:\mathbb{N}\rightarrow \{-1,0,1\}$ be the Mobius function, meaning that $\mu(n)$ is $0$ if $n$ is not square-free and is $(-1)^k$ if $n$ is square-free with $k$ distinct prime factors.  First assume that $d \equiv 1 \bmod 4$, so we are considering
even fundamental discriminants $\{d \le X: d \equiv 1 \bmod 4,
\mu(d)^2=1\}$; it is trivial to modify the arguments below for $d$
such that $d/4 \equiv 2$ or $3$ modulo $4$ and $\mu(d/4)^2=1$. Let
$\chi_4(n)$ be the non-trivial character modulo 4: $\chi_4(2m) = 0$
and \be \twocase{\chi_4(n) \ = \ }{1}{if $n\equiv 1 \bmod 4$}{0}{if
$n \equiv 3 \bmod 4$.} \ee We have \bea\label{eq:startSXanalysis}
S(X) & \ = \ & \sum_{d\le X \atop \mu(d)^2=1,\ d\equiv 1 \bmod 4} 1
\nonumber\\ & \ = \ & \sum_{d \le
X \atop 2 \notdiv d} \mu(d)^2 \cdot \frac{1 + \chi_4(d)}2 \nonumber\\
&=& \foh \sum_{d\le X \atop 2 \notdiv d} \mu(d)^2 + \foh \sum_{d\le
X} \mu(d)^2 \chi_4(d) \ = \ S_1(X) + S_2(X).\eea By M\"obius
inversion \be \twocase{\sum_{m^2|d} \mu(m) \ = \ }{1}{if $d$ is
square-free}{0}{otherwise.} \ee Thus \bea S_1(X) & \ = \ & \foh
\sum_{d \le X \atop 2 \notdiv d} \sum_{m^2|d} \mu(m) \nonumber\\ &=&
\foh \sum_{m \le X^{1/2}\atop 2\notdiv m} \mu(m) \cdot \sum_{d\ \le\
X/m^2 \atop 2\notdiv d} 1 \nonumber\\ &=& \foh \sum_{m\le X^{1/2}
\atop 2\notdiv m} \mu(m)\left(\frac{X}{2m^2} + O(1)\right)
\nonumber\\ &=& \frac{X}{4} \sum_{m=1 \atop 2\notdiv m}^\infty
\frac{\mu(m)}{m^2} + O(X^{1/2}) \nonumber\\ &=& \fof
\frac{6}{\zeta(2)} \cdot \left(1 - \frac1{2^2}\right)^{-1}\cdot X +
O(X^{1/2}) \nonumber\\ &=& \frac{2}{\pi^2}X + O(X^{1/2}) \eea
(because we are missing the factor corresponding to $2$ in
$1/\zeta(2)$ above). Arguing in a similar manner shows $S_2(X) =
O(X^{1/2})$; this is due to the presence of $\chi_4$, giving us \be
S_2(X) \ = \ \foh \sum_{m \le X^{1/2}} \chi_4(m^2)\mu(m) \sum_{d\le
X/m^2} \chi_4(d) \ \ll \ X^{1/2} \ee (because we are summing
$\chi_4$ at consecutive integers, and thus this sum is at most 1). A
similar analysis shows that the number of even fundamental
discriminants $d\le X$ with $d/4 \equiv 2$ or $3$ modulo $4$ is
$X/\pi^2 + O(X^{1/2})$. Thus \be \sum_{d\le X \atop d\ {\rm an\
even\ fund.\ disc.}}1 \ = \ X^\ast \ = \
\frac{3}{\pi^2}X+O(X^{1/2}).\ee

We may trivially modify the above calculations to determine the
number of even fundamental discriminants $d\le X$ with $p|d$ for a
fixed prime $p$. We first assume $p\equiv 1 \bmod 4$. In
\eqref{eq:startSXanalysis} we replace $\mu(d)^2$ with $\mu(pd)^2$,
$d \le X$ with $d \le X/p$, $2\notdiv d$ and $(2p,d)=1$. These imply
that $d \le X$, $p|d$ and $p^2$ does not divide $d$. As $d$ and $p$
are relatively prime, $\mu(pd) = \mu(p)\mu(d)$ and the main term
becomes \bea S_{1;p}(X) & \ = \ &
\foh \sum_{d \le X/p \atop (2p,d)=1} \sum_{m^2|d} \mu(m) \nonumber\\
&=& \foh \sum_{m \le (X/p)^{1/2}\atop (2p,m)=1} \mu(m) \cdot
\sum_{d\ \le\ (X/p)/m^2 \atop (2p,d)=1} 1 \nonumber\\ &=& \foh
\sum_{m\le (X/p)^{1/2} \atop
(2p,m)=1} \mu(m)\left(\frac{X/p}{m^2} \cdot \frac{p-1}{2p} + O(1)\right) \nonumber\\
&=& \frac{(p-1)X}{4p^2} \sum_{m=1 \atop (2p, m)=1}^\infty
\frac{\mu(m)}{m^2} + O(X^{1/2}) \nonumber\\ &=& \fof
\frac{6}{\zeta(2)} \cdot \left(1 - \frac1{2^2}\right)^{-1} \cdot
\left(1-\frac1{p^2}\right)^{-1} \cdot \frac{(p-1)X}{p^2} + O(X^{1/2}) \nonumber\\
&=& \frac{2X}{(p+1)\pi^2} + O(X^{1/2}), \eea and the cardinality of
this piece is reduced by $(p+1)^{-1}$ (note above we used $\#\{n \le
Y: (2p,n)=1\}$ $=$ $\frac{p-1}{2p}Y+O(1)$). A similar analysis holds
for $S_{2;p}(X)$, as well as the even fundamental discriminants $d$
with $d/4 \equiv 2$ or $3$ modulo $4$).

We need to trivially modify the above arguments if $p\equiv 3 \bmod
4$. If for instance we require $d\equiv 1 \bmod 4$ then instead of
replacing $\mu(d)^2$ with $\mu(d)^2 (1 + \chi_4(d))/2$ we replace it
with $\mu(pd)^2 (1-\chi_4(d))/2$, and the rest of the proof proceeds
similarly.

For $p=2$, we have a different situation. Note if $d\equiv 1
\bmod 4$ then 2 {never} divides $d$, while if $d/4 \equiv 2$ or
3 modulo 4 then 2 {always} divides $d$. There are $3X/\pi^2 +
o(X^{1/2})$ even fundamental discriminants at most $X$, and $X/\pi^2
+ O(x^{1/2})$ of these are divisible by 2. Thus, if our family is
all even fundamental discriminants, we do get the factor of
$1/(p+1)$ for $p=2$, as one-third (which is $1/(2+1)$ of the
fundamental discriminants in this family are divisible by $2$.
\end{proof}

\begin{lemma}\label{secondmillerlemma} Let $d$ denote an
even fundamental discriminant at most $X$ and $X^\ast = \sum_{d\le
X} 1$ and let $z = \nu - i w \frac{\log (X/2\pi)}{\pi}$ with $w=\frac{1}{4}-\epsilon$ ($\epsilon>0$ small).
Then \be \sum_{d\le X} e^{-2\pi i z\frac{\log(d/2\pi)}{\log (X/2\pi)}} \ = \ X^*e^{-2\pi i z}\left(1-\frac{2\pi i z}{\log(X/2\pi)}\right)^{-1}+O(X^{2\epsilon}). \ee
\end{lemma}

\begin{proof}  We may rewrite our sum as
\begin{align}
\sum_{d\leq X}e^{-2\pi  i z\frac{\log(d/2\pi)}{L}}\nonumber
\ = \ &\sum_{d\leq X}e^{-2\pi  i z\frac{(\log d-\log 2\pi)}{L}}\nonumber
\\ \ = \ &\sum_{d\leq X}e^{-2\pi  i z\frac{\log d}{L}}e^{2\pi  i z\frac{\log 2\pi}{L}}\nonumber
\\ \ = \ &e^{2\pi  i z\frac{\log 2\pi}{L}}\sum_{d\leq X}d^{-2\pi i z/L}.\label{tobepartialed}
\end{align}
Recall the integral version of partial summation, namely that if $h(x)$ is a continuously differentiable function and $A(x)=\sum_{n\leq x}a_n$, then
\be \sum_{n\leq x}a_nh(n)=A(x)h(x)-\int_1^x A(u)h'(u)du.
\ee Considering $h(x)=x^{-2\pi i z/L}$ and $a_n=1$ for $n$ an even fundamental discriminant and $a_n=0$ otherwise, by Lemma
\ref{firstmillerlemma} we have \be A(u)\ = \ \sum_{d\le u} 1 \ = \
\frac{3u}{\pi^2} + O(u^{1/2}). \ee
This allows us to rewrite equation \eqref{tobepartialed} as
\begin{align}
&e^{2\pi  i z\frac{\log 2\pi}{L}}\sum_{d\leq X}d^{-2\pi i z/L}\nonumber
\\ =  &e^{2\pi  i z\frac{\log 2\pi}{L}}\left[\left(\frac{3X}{\pi^2}+O(X^{1/2})\right)X^{-2\pi i z/L}-\int_{1}^X \left(\frac{3u}{\pi^2} + O(u^{1/2})\right) \cdot u^{-2\pi i z/L}u^{-1}\left(\frac{-2\pi i z}{L}\right)du\right]
\end{align}
Note that we may rewrite $\frac{-2\pi i z}{L}=\frac{-2\pi i (\nu - i w L/\pi)}{L}=-2 w+i\delta$, where $\delta\in\mathbb{R}$.  As order of magnitude depends only on the real part of the exponent, we have $O(X^{1/2})X^{-2\pi i z/L}=)(X^{1/2-2w})=O(X^{2\epsilon})$ due to our choice of $w$.  Similarly, we may write $O(u^{1/2})u^{-2\pi i z/L}u^{-1}=O(u^{1/2-2w-1})=O(u^{-1+2\epsilon})$, meaning the integral (from $1$ to $X$) of that term ends up as $O(X^{2\epsilon})$.  As multiplying these terms by the $e^{2\pi i z\frac{\log 2\pi}{L}}$ only decreases their order of magnitude in $X$, we have
\begin{align}
&e^{2\pi  i z\frac{\log 2\pi}{L}}\sum_{d\leq X}d^{-2\pi i z/L}\nonumber
\\ =  &e^{2\pi  i z\frac{\log 2\pi}{L}}\left[\frac{3X}{\pi^2}X^{-2\pi i z/L}-\int_{1}^X \frac{3u}{\pi^2} \cdot u^{-2\pi i z/L}u^{-1}\left(\frac{-2\pi i z}{L}\right)du\right]+O(X^{2\epsilon})\nonumber
\\ \ = \ &e^{2\pi  i z\frac{\log 2\pi}{L}}\left[\frac{3}{\pi^2}X^{1-2\pi i z/L}+\frac{\frac{3}{\pi^2}\cdot 2\pi i z}{L}\int_{1}^X u^{-2\pi i z/L}du\right]+O(X^{2\epsilon})\nonumber
\\ \ = \ &e^{2\pi  i z\frac{\log 2\pi}{L}}\left[\frac{3}{\pi^2}X^{1-2\pi i z/L}+\frac{\frac{3}{\pi^2}\cdot 2\pi i z}{L}\frac{X^{1-2\pi i z/L}}{1-2\pi i z/L}\right]+O(X^{2\epsilon})\nonumber
\\ \ = \ &\frac{3}{\pi^2}X^{1-2\pi i z/L}e^{2\pi  i z\frac{\log 2\pi}{L}}\left[1+\frac{2\pi i z}{L}\sum_{k=0}^\infty\left(\frac{L}{2\pi i z}\right)^k\right]+O(X^{2\epsilon})\nonumber
\\ \ = \ &\frac{3}{\pi^2}Xe^{-2\pi i z\left(\frac{\log X}{L}-\frac{\log 2\pi }{L}\right)}\left(1-\frac{2\pi i z}{L}\right)^{-1}+O(X^{2\epsilon})\nonumber
\\ \ = \ &X^*e^{-2\pi i z}\left(1-\frac{2\pi i z}{\log(X/2\pi)}\right)^{-1}+O(X^{2\epsilon}).
\end{align}

\end{proof}

\begin{lemma}\label{lem:ntsumkeven1} Notation as in Section 2, we have \bea S_{{\rm even};1} & \ =
\ &  \frac{g(0)}2 + \frac{1}{L}\intii
g(\nu) \left(\frac{L'}{L}\left(1+\frac{2\pi i \nu}{L},{\rm sym}^2\Delta\right) - \frac{\zeta'}{\zeta}\left(1+\frac{2\pi i \nu}{L}\right)\right) d\nu
\nonumber\\ \eea
\end{lemma}

\begin{proof}

Let \be \twocase{\Lambda_\Delta(n)\ =\ }{(\alpha_p^{2\ell} + \overline{\alpha}_p^{2\ell}) \log p}{if $n=p^\ell$}{0}{otherwise.} \ee Note this is the natural generalization of $\Lambda(n)$ for the tau curve.

We have \be S_{{\rm even};1} \ = \ -\frac{1}{L}\sum_{n=1}^\infty \frac{\Lambda_\Delta(n)}{n} \ \hg\left(\frac{\log
n}{L}\right). \ee We use Perron's formula to re-write $S_{{\rm
even};1}$ as a contour integral. For any $\gep > 0$ set 
\bea\label{eq:integraldefI1} I_1 \ =
\ \ci \int_{\Re(z)=1+\gep} g\left(\frac{(2z-2)\log A}{2\pi i}\right)
\sum_{n=1}^\infty \frac{\Lambda_\Delta(n)}{n^z}\ dz; \eea
 we will later take $A = X/2\pi$. We write $z = 1+\gep+iy$ and write
$g(x+iy)$ in terms of the integral of $\hg(u)$, giving us 
\bea I_1 & \
= \ & \sum_{n=1}^\infty \frac{\Lambda_\Delta(n)}{n^{1+\gep}} \ci \intii
g\left(\frac{y\log A}{\pi}-\frac{i\gep\log
A}{\pi}\right)e^{-iy\log n} idy \nonumber\\ &=& \sum_{n=1}^\infty
\frac{\Lambda_\Delta(n)}{n^{1+\gep}} \frac1{2\pi} \intii \left[\intii
\left[\hg(u)e^{\gep u \log A}\right] e^{-2\pi i \frac{-y\log A}{\pi
}u} du \right]e^{-iy\log n}dy.\ \ \ \ \eea 
We let $h_\gep(u) =
\hg(u) e^{\gep u \log A}$. Note that $h_\gep$ is a smooth, compactly
supported function and $\widehat{\widehat{h_\gep}}(w) = h_\gep(-w)$.
Thus \bea I_1 & \ = \ & \sum_{n=1}^\infty
\frac{\Lambda_\Delta(n)}{n^{1+\gep}}  \frac1{2\pi} \intii
\widehat{h_\gep}\left(-\frac{y\log A}{\pi}\right) e^{-iy\log n} dy
\nonumber
\\ &=& \sum_{n=1}^\infty \frac{\Lambda_\Delta(n)}{n^{1+\gep}}
\frac1{2\pi}\intii \widehat{h_\gep}(y) e^{-2\pi i \frac{-y \log
n}{\log A}}\ \frac{\pi dy}{\log A} \nonumber
\\ &=&
\sum_{n=1}^\infty \frac{\Lambda_\Delta(n)}{n^{1+\gep}} \frac1{\log A}\
\widehat{\widehat{h_\gep}}\left(-\frac{\log n}{\log A}\right)
\nonumber\\
&=& \sum_{n=1}^\infty \frac{\Lambda_\Delta(n)}{n^{1+\gep}}  \frac1{\log A}\
\hg\left(\frac{\log n}{\log A}\right) e^{\gep \log n} \nonumber\\
&=& \frac1{\log A}\sum_{n=1}^\infty \frac{\Lambda_\Delta(n)}{n}\
\hg\left(\frac{\log n}{\log A}\right). \eea
 By taking $A = X/2\pi$
we find \be S_{{\rm even};1} \ = \ -\frac{1}{L}\sum_{n=1}^\infty \frac{\Lambda_\Delta(n)}{n} \ \hg\left(\frac{\log
n}{L}\right) \ = \ -I_1.\ee

We now re-write $I_1$ by shifting contours; we will not pass any
poles as we shift. For each $\delta > 0$ we consider the contour
made up of three pieces: $(1-i\infty,1-i\delta]$, $C_\delta$, and
$[1-i\delta,1+i\infty)$, where $C_\delta = \{z: z-1 = \delta
e^{i\theta}, \theta \in [-\pi/2,\pi/2]\}$ is the semi-circle going
counter-clockwise from $1-i\delta$ to $1+i\delta$. By Cauchy's
residue theorem, we may shift the contour in $I_1$ from $\Re(z) =
1+\gep$ to the three curves above. 

For use in rewriting the integral, we will consider the logarithmic derivative of the symmetric square $L$-function attached to $\tau^\ast$. From (3.15) of \cite{ILS} (recall the level $N=1$ in our case) it is \be L(s, {\rm sym}^2 \Delta) \ = \ \prod_p \left(1 - \frac{\alpha_p^2}{p^s}\right)^{-1} \left(1 - \frac{1}{p^s}\right)^{-1} \left(1 - \frac{\overline{\alpha}^2_p}{p^s}\right)^{-1}, \ee as $\alpha_p \overline{\alpha}_p = 1$. Taking the logarithmic derivative yields \be \frac{L'(s,{\rm sym}^2 \Delta)}{L(s,{\rm sym^2}\Delta)} \ = \ \sum_{\ell=1}^\infty \frac{(\alpha_p^{2\ell} + 1+\overline{\alpha}_p^{2\ell})\log p}{p^{s\ell}}, \ee so 
\bea \sum_{\ell=1}^\infty \frac{(\alpha_p^{2\ell} +\overline{\alpha}_p^{2\ell})\log p}{p^{s\ell}}&  \ = \ & \frac{L'(s,{\rm sym}^2 \Delta)}{L(s,{\rm sym^2}\Delta)} -  \sum_{\ell=1}^\infty \frac{\log p}{p^{s\ell}} \nonumber\\ &=& \frac{L'(s,{\rm sym}^2 \Delta)}{L(s,{\rm sym^2}\Delta)} -  \sum_n \frac{\Lambda(n)}{n^s} \nonumber\\ &=& \frac{L'(s,{\rm sym}^2 \Delta)}{L(s,{\rm sym^2}\Delta)} -  \frac{\zeta'(s)}{\zeta(s)}. \eea

We shall use this in replacing $\sum_n \Lambda_\Delta(n)n^{-z}$ in the integral definition of $I_1$ in \eqref{eq:integraldefI1}. We find \bea I_1 &\ = \ &
\ci\left[\int_{1-i\infty}^{1-i\delta} + \int_{C_\delta} +
\int_{1+i\delta}^{1+i\infty} g\left(\frac{(2z-2)\log A}{2\pi
i}\right) \sum_n\frac{-\Lambda_\Delta(n)}{n^z}\ dz\right] \nonumber\\
&=& \ci\left[\int_{1-i\infty}^{1-i\delta} + \int_{C_\delta} +
\int_{1+i\delta}^{1+i\infty} g\left(\frac{(2z-2)\log A}{2\pi
i}\right) \left(-\frac{L'}{L}(z,{\rm sym}^2\Delta) + \frac{\zeta'}{\zeta}(z)\right)\ dz\right].\nonumber\\ \eea 

The integral over $C_\delta$ is easily evaluated. Shimura \cite{Sh} proved that $L(s,{\rm sym}^2 \Delta)$ is entire, and thus so too is its logarithmic derivative. 
 Thus there is no contribution from the symmetric square piece in the limit as $\delta \to 0$. 
  We now argue that the $\zeta'/\zeta$ term contributes $-g(0)/2$. As $\zeta(s)$ has a pole at $s=1$, $\zeta'(s)/\zeta(s) = -1/(s-1) + \cdots$, and thus we must multiply the contribution from the residue by $-1$ because of the pole. We get just minus half the residue of $g\left(\frac{(2z-2)\log
A}{2\pi i}\right)$. Thus the $C_\delta$ piece is
$-g(0)/2$. 

We now take the limit as $\delta \to 0$: \be I_1 \ = \
-\frac{g(0)}2 - \lim_{\delta \to 0} \frac1{2\pi}
\left[\int_{-\infty}^{-\delta} + \int_{\delta}^\infty
g\left(\frac{y\log A}{\pi}\right) \
\left(\frac{L'}{L}(z,{\rm sym}^2\Delta) - \frac{\zeta'}{\zeta}(z)\right) dy\right].\ee As $g$ is an even
Schwartz function, the limit of the integral above is well-defined
(for large $y$ this follows from the decay of $g$, while for small
$y$ it follows from the fact that $\zeta'(1+iy)/\zeta(1+iy)$ has a
simple pole at $y=0$ and $g$ is even). We again take $A=X/2\pi$,
and change variables to $\nu = \frac{y\log A}{\pi} = \frac{yL
}{\pi}$. Thus \bea I_1 & \ = \ & -\frac{g(0)}2 - \frac{1}{L}\intii
g(\nu) \left(\frac{L'}{L}\left(1+\frac{2\pi i \nu}{L},{\rm sym}^2\Delta\right) - \frac{\zeta'}{\zeta}\left(1+\frac{2\pi i \nu}{L}\right)\right) d\nu \nonumber\\ &=& -S_{{\rm even},1}, \eea which completes the proof of Lemma
\ref{lem:ntsumkeven1}. \end{proof}

\begin{lemma}\label{sumlemma}  Let $y'=y-iw$ be as in Proposition \ref{bconverges}.  Then for all $\epsilon>0$,
\begin{itemize}

\item[i.)]  \be \sum_{m=1}^\infty\frac{\tau^*(p^{2m})}{p^{m(1-2iy')}} \ =\ \frac{\tau^*(p^2)}{p^{1-2iy'}}+O\left(\frac{1}{p^{2-4w-\epsilon}}\right)
\ee

\item[ii.)]\be \frac{\tau^*(p)}{p}\sum_{m=0}^\infty\frac{\tau^*(p^{2m+1})}{p^{m(1-2iy')}} \ = \ \frac{\tau^*(p)^2}{p^{2}}+O\left(\frac{1}{p^{2-2w-\epsilon}}\right)
\ee

\item[iii.)]\be \frac{1}{p^{1+2iy'}}\sum_{m=0}^\infty\frac{\tau^*(p^{2m})}{p^{m(1-2iy')}} \ = \ \frac{1}{p^{1+2iy'}}+O\left(\frac{1}{p^{2-\epsilon}}\right).
\ee

\end{itemize}
\end{lemma}

\begin{proof} Fix $\epsilon$, and pick $C$ such that $|\tau^*(n)|\leq Cn^\epsilon$ for all $n$ (such a $C$ exists by Deligne's theorem, which implies $\tau(n)=O(n^{11/2+\epsilon})$ for all $\epsilon>0$; see \cite{Serre}).  Also, recall that for real numbers $a,x,y$, we have $|a^{x+iy}|=|a^x|$.

\begin{itemize}

\item[i.)]  The first claim follows from

\begin{align}
\left| \sum_{m=2}^\infty\frac{\tau^*(p^{2m})}{p^{m(1-2iy')}}\right| \ \leq\ &\sum_{m=2}^\infty\left|\frac{C\cdot(p^{2m})^\epsilon}{p^{m(1-2iy')}}\right|\nonumber
\\ \ = \ &C\sum_{m=2}^\infty\frac{1}{\left|p^{1-2iy'-2\epsilon}\right|^m}\nonumber
\\ \ = \ &C\sum_{m=2}^\infty\frac{1}{\left|p^{1-2w-2iy-2\epsilon}\right|^m}\nonumber
\\ \ = \ &C\sum_{m=2}^\infty\left(\frac{1}{p^{1-2w-2\epsilon}}\right)^m\nonumber
\\ \ = \ &C\left(\frac{1}{p^{1-2w-2\epsilon}}\right)^2\sum_{m=0}^\infty\left(\frac{1}{p^{1-2w-2\epsilon}}\right)^m\nonumber
\\ \ = \ &C\left(\frac{1}{p^{1-2w-2\epsilon}}\right)^2\cdot\frac{1}{1-\frac{1}{p^{1-2w-1\epsilon}}}\nonumber
\\ \ = \ &C'\frac{1}{p^{2-4w-4\epsilon}}\ = \ O\left(\frac{1}{p^{2-4w-4\epsilon}}\right).
\end{align}

\item[ii.)]  The second claim follows from
\begin{align}
\left|\frac{\tau^*(p)}{p}\sum_{m=1}^\infty\frac{\tau^*(p^{2m+1})}{p^{m(1-2iy')}}\right|\ \leq\  &\frac{|\tau^*(p)|}{p}\sum_{m=1}^\infty\left|\frac{C\cdot(p^{2m+1})^\epsilon}{p^{m(1-2iy')}}\right|\nonumber
\\ \ = \ &\frac{C|\tau^*(p)|}{p^{1-\epsilon}}\sum_{m=1}^\infty\frac{1}{\left|p^{1-2iy'-2\epsilon}\right|^m}\nonumber
\\ \ = \ &\frac{C|\tau^*(p)|}{p^{1-\epsilon}}\sum_{m=1}^\infty\frac{1}{\left|p^{1-2w-2iy-2\epsilon}\right|^m}\nonumber
\\ \ = \ &\frac{C|\tau^*(p)|}{p^{1-\epsilon}}\sum_{m=1}^\infty\left(\frac{1}{p^{1-2w-2\epsilon}}\right)^m\nonumber
\\ \ = \ &\frac{C|\tau^*(p)|}{p^{1-\epsilon}}\frac{1}{p^{1-2w-2\epsilon}}\sum_{m=0}^\infty\left(\frac{1}{p^{1-2w-2\epsilon}}\right)^m\nonumber
\\ \ = \ &\frac{C|\tau^*(p)|}{p^{2-2w-3\epsilon}}\frac{1}{1-\frac{1}{p^{1-2w-1\epsilon}}}\nonumber
\\ \ \leq\ &C'\frac{1}{p^{2-2w-3\epsilon}}=O\left(\frac{1}{p^{2-2w-3\epsilon}}\right).
\end{align}

\item[iii.)] The third claim follows from the fact that $\tau^*(1)=1$ and from the following bound (where several steps are omitted due to similarity of the bound for the first claim):
\begin{align}
\left|\frac{1}{p^{1+2iy'}} \sum_{m=1}^\infty\frac{\tau^*(p^{2m})}{p^{m(1-2iy')}}\right|\ \leq\ & \frac{1}{|p^{1+2iy'}|}\sum_{m=1}^\infty\left|\frac{C\cdot(p^{2m})^\epsilon}{p^{m(1-2iy')}}\right|\nonumber
\\ \ = \ &\frac{C}{p^{1+2w}}\sum_{m=1}^\infty\frac{1}{\left|p^{1-2iy'-2\epsilon}\right|^m}\nonumber
\\ \ = \ & \frac{C}{p^{1+2w}}\frac{1}{p^{1-2w-2\epsilon}}\sum_{m=0}^\infty\left(\frac{1}{p^{1-2w-2\epsilon}}\right)^m\nonumber
\\ \ = \ &\frac{C}{p^{1+2w}}\frac{1}{p^{1-2w-2\epsilon}}\cdot\frac{1}{1-\frac{1}{p^{1-2w-1\epsilon}}}\nonumber
\\ \leq\ &C'\frac{1}{p^{2-4\epsilon}}=O\left(\frac{1}{p^{2-4\epsilon}}\right).
\end{align}
\end{itemize}
\end{proof}


\newpage

\begin{thebibliography}{GJMMNPP}

\bibitem[Be]{Be}
\newblock M. V. Berry, \emph{Semiclassical formula for the number
variance of the Riemann zeros}, Nonlinearity \textbf{1} (1988),
399--407.

\bibitem[BeKe]{BeKe}
\newblock M. V. Berry and J. P. Keating, \emph{The Riemann zeros and
eigenvalue asymptotics}, Siam Review \textbf{41} (1999), no. 2,
236--266.

\bibitem[BBLM]{BBLM}
\newblock E. Bogomolny, O. Bohigas, P. Leboeuf and A. G. Monastra,
\emph{On the spacing distribution of the Riemann zeros: corrections
to the asymptotic result}, Journal of Physics A: Mathematical and
General \textbf{39} (2006), no. 34, 10743--10754.

\bibitem[BoKe]{BoKe}
\newblock E. B. Bogomolny and J. P. Keating, \emph{Gutzwiller's
trace formula and spectral statistics: beyond the diagonal
approximation}, Phys. Rev. Lett. \textbf{77} (1996), no. 8,
1472--1475.

\bibitem[BCDT]{BCDT}
\newblock C. Breuil, B. Conrad, F. Diamond and R. Taylor, \emph{On
the modularity of elliptic curves over \textbf{Q}: wild $3$-adic
exercises}, J. Amer. Math. Soc. \textbf{14} (2001), no. 4, 2001,
843--939.

\bibitem[Co]{matrix}
J.B. Conrey, \emph{Notes on eigenvalue distributions for the classical compact groups}, London Math. Soc. Lecture Note Ser., 322, Cambridge Univ. Press, Cambridge, 2005.
\label{matrix}


\bibitem[CF]{CF}
\newblock B. Conrey and D. Farmer, \emph{Mean values of
$L$-functions and symmetry}, Internat. Math. Res. Notices 2000, no.
17, 883--908.

\bibitem[CFKRS]{CFKRS}
\newblock B. Conrey, D. Farmer, P. Keating, M. Rubinstein and N.
Snaith, \emph{Integral moments of $L$-functions}, Proc. London Math.
Soc. (3)  \textbf{91} (2005),  no. 1, 33--104.

\bibitem[CFZ1]{CFZ1}
\newblock J. B. Conrey, D. W. Farmer and M. R. Zirnbauer, \emph{Autocorrelation of ratios
of $L$-functions}, preprint. \texttt{http://arxiv.org/abs/0711.0718}

\bibitem[CFZ2]{CFZ2}
\newblock J. B. Conrey, D. W. Farmer and M. R. Zirnbauer, \emph{Howe pairs, supersymmetry,
and ratios of random characteristic polynomials for the classical
compact groups}, preprint.
\texttt{http://arxiv.org/abs/math-ph/0511024}

\bibitem[CS1]{CS1}
\newblock J. B. Conrey and N. C. Snaith, \emph{Applications of the
$L$-functions Ratios Conjecture},  Proc. Lon. Math. Soc. \textbf{93} (2007), no 3, 594--646.

\bibitem[CS2]{CS2}
\newblock J. B. Conrey and N. C. Snaith, \emph{Triple correlation of
the Riemann zeros}, preprint.
\texttt{http://arxiv.org/abs/math/0610495}

\bibitem[Da]{Da}
\newblock H. Davenport, \emph{Multiplicative Number Theory, $2$nd edition},
 Graduate Texts in Mathematics \textbf{74}, Springer-Verlag, New York,
 $1980$, revised by H. Montgomery.

\bibitem[DHKMS1]{DHKMS1}
\newblock E. Due\~nez, D. K. Huynh, J. P. Keating, S. J. Miller and
N. C. Snaith, \emph{The lowest eigenvalue of Jacobi Random Matrix Ensembles and Painlev\'e VI}, preprint.

\bibitem[DHKMS2]{DHKMS2}
\newblock E. Due\~nez, D. K. Huynh, J. P. Keating, S. J. Miller and
N. C. Snaith, \emph{A random matrix model for elliptic curve $L$-functions of finite conductor}, preprint.

\bibitem[DM1]{DM1}
\newblock E. Due\~nez and S. J. Miller, \emph{The low lying zeros of a
$\text{GL}(4)$ and a $\text{GL}(6)$ family of $L$-functions},
Compositio Mathematica \textbf{142} (2006), no. 6, 1403--1425.

\bibitem[DM2]{DM2}
\newblock E. Due\~nez and S. J. Miller, \emph{The effect of
convolving families of $L$-functions on the underlying group
symmetries}, preprint. \texttt{http://arxiv.org/abs/math/0607688}

\bibitem[Dy1]{Dy1}
F. Dyson, \emph{Statistical theory of the energy levels of complex
systems: I, II, III}, J. Mathematical Phys. \textbf{3} (1962)
140--156, 157--165, 166--175.

\bibitem[Dy2]{Dy2}
F. Dyson, \emph{The threefold way. Algebraic structure of symmetry
groups and ensembles in quantum mechanics}, J. Mathematical Phys.,
\textbf{3} (1962) 1199--1215.


\bibitem[ET]{ET}
\newblock A. Erd$\acute{{\rm e}}$lyi and F. G. Tricomi, \emph{The asymptotic expansion of a ratio
of gamma functions}, Pacific J. Math. \textbf{1} (1951), no. 1,
133--142.

\bibitem[FM]{FM}
F. W. K. Firk and S. J. Miller, \emph{Nuclei, Primes and the Random Matrix Connection}, Symmetry \textbf{1} (2009), 64--105.

\bibitem[FI]{FI}
\newblock E. Fouvry and H. Iwaniec, \emph{Low-lying zeros of dihedral
$L$-functions}, Duke Math. J.  \textbf{116} (2003),  no. 2, 189-217.

\bibitem[Gao]{Gao}
\newblock P. Gao, \emph{$N$-level density of the low-lying zeros of
quadratic Dirichlet $L$-functions}, Ph.~D thesis, University of
Michigan, 2005.

\bibitem[GJMMNPP]{GJMMNPP}
\newblock J. Goes, S. Jackson, S. J. Miller, D. Montague, K. Ninsuwan, R. Peckner and T. Pham, \emph{A unitary test of the $L$-functions Ratios Conjecture}, to appear in the Journal of Number Theory.

\bibitem[GM]{GM}
\newblock J. Goes and S. J. Miller, \emph{Towards an `average' version of the Birch and Swinnerton-Dyer Conjecture}, preprint.

\bibitem[G\"u]{Gu}
\newblock A. G\"ulo\u{g}lu, \emph{Low-Lying Zeros of Symmetric
Power $L$-Functions}, Internat. Math. Res. Notices 2005, no. 9,
517-550.

\bibitem[HW]{HW}
\newblock G. Hardy and E. Wright, \emph{An Introduction to the
Theory of Numbers}, fifth edition, Oxford Science Publications,
Clarendon Press, Oxford, $1995$.

\bibitem[Ha]{Ha}
B. Hayes, \emph{The spectrum of Riemannium}, American Scientist
\textbf{91} (2003), no. 4, 296--300.


\bibitem[Hej]{Hej}
\newblock D. Hejhal, \emph{On the triple correlation of zeros of
the zeta function}, Internat. Math. Res. Notices 1994, no. 7,
294-302.

\bibitem[HM]{HM}
\newblock C. Hughes and S. J. Miller, \emph{Low-lying zeros of $L$-functions
with orthogonal symmtry}, Duke Math. J., \textbf{136}
(2007), no. 1, 115--172.

\bibitem[HR]{HR}
\newblock C. Hughes and Z. Rudnick, \emph{Linear Statistics of
Low-Lying Zeros of $L$-functions},  Quart. J. Math. Oxford
\textbf{54} (2003), 309--333.

\bibitem[HKS]{HKS}
\newblock D. K. Huynh, J. P. Keating and N. C. Snaith, work in
progress.

\bibitem[HMM]{HMM}
\newblock D. K. Huynh, S. J. Miller and
R. Morrison, \emph{An Elliptic Curve Test of the $L$-Functions Ratios Conjecture}.

\bibitem[IK]{IK}
H. Iwaniec and E. Kowalski, \emph{Analytic Number Theory}, AMS
Colloquium Publications, Vol. 53, AMS, Providence, RI, $2004$.


\bibitem[ILS]{ILS}
\newblock H. Iwaniec, W. Luo and P. Sarnak, \emph{Low lying zeros of
families of $L$-functions}, Inst. Hautes Études Sci. Publ. Math.
\textbf{91}, 2000, 55--131.

\bibitem[Ju1]{Ju1}
\newblock M. Jutila, \emph{On character sums and class numbers},
Journal of Number Theory \textbf{5} (1973), 203--214.

\bibitem[Ju2]{Ju2}
\newblock M. Jutila, \emph{On mean values of Dirichlet polynomials
with real characters}, Acta Arith. \textbf{27} (1975), 191--198.

\bibitem[Ju3]{Ju3}
\newblock M. Jutila, \emph{On the mean value of $L(1/2,\chi)$ for real
characters}, Analysis \textbf{1} (1981), no. 2, 149--161.

\bibitem[KaSa1]{KaSa1}
\newblock N.~Katz and P.~Sarnak, \emph{Random Matrices, Frobenius
Eigenvalues and Monodromy}, AMS Colloquium Publications \textbf{45},
AMS, Providence, $1999$.

\bibitem[KaSa2]{KaSa2}
\newblock N.~Katz and P.~Sarnak, \emph{Zeros of zeta functions and symmetries},
Bull. AMS \textbf{36}, $1999$, $1-26$.

\bibitem[Ke]{Ke}
\newblock J. P. Keating, \emph{Statistics of quantum eigenvalues and the Riemann zeros},
in Supersymmetry and Trace Formulae: Chaos and Disorder, eds. I. V.
Lerner, J. P. Keating \& D. E Khmelnitskii (Plenum Press), 1--15.

\bibitem[KeSn1]{KeSn1}
\newblock J. P. Keating and N. C. Snaith, \emph{Random matrix theory
and $\zeta(1/2+it)$},  Comm. Math. Phys.  \textbf{214} (2000),  no.
1, 57--89.

\bibitem[KeSn2]{KeSn2}
\newblock J. P. Keating and N. C. Snaith, \emph{Random matrix theory and $L$-functions at $s=1/2$},
Comm. Math. Phys.  \textbf{214}  (2000),  no. 1, 91--110.

\bibitem[KeSn3]{KeSn3}
\newblock J. P. Keating and N. C. Snaith, \emph{Random
matrices and $L$-functions}, Random matrix theory, J. Phys. A
\textbf{36} (2003), no. 12, 2859--2881.

\bibitem[Mil1]{Mil1}
\newblock S. J. Miller, \emph{$1$- and $2$-level densities for families of elliptic
curves: evidence for the underlying group symmetries}, Compositio
Mathematica \textbf{140} (2004), 952--992.

\bibitem[Mil2]{Mil2}
\newblock S. J. Miller, \emph{Variation in the number of points on elliptic curves and
applications to excess rank}, C. R. Math. Rep. Acad. Sci. Canada
\textbf{27} (2005), no. 4, 111--120.

%\bibitem[Mil3]{Mil3}
%\newblock S. J. Miller, \emph{Investigations of zeros near the central point
%of elliptic curve $L$-functions}, Experimental Mathematics
%\textbf{15} (2006), no. 3, 257--279.

\bibitem[Mil3]{Mil3}
\newblock S. J. Miller, \emph{A symplectic test of the $L$-Functions Ratios Conjecture}, Int Math Res Notices (2008) Vol. 2008, article ID rnm146, 36 pages, doi:10.1093/imrn/rnm146.

\bibitem[Mil4]{Mil4}
\newblock S. J. Miller, \emph{Lower
order terms in the $1$-level density for families of holomorphic
cuspidal newforms}, Acta Arithmetica \textbf{137} (2009), 51--98.

\bibitem[Mil5]{Mil5}
\newblock S. J. Miller, \emph{An orthogonal test of the $L$-Functions Ratios Conjecture}, Proceedings of the London Mathematical Society 2009, doi:10.1112/plms/pdp009.

\bibitem[MilMo]{MilMo}
\newblock S. J. Miller and D. Montague, \emph{An Orthogonal Test of the $L$-functions Ratios Conjecture, II}, preprint.

\bibitem[MT-B]{MT-B}
\newblock S. J. Miller and R. Takloo-Bighash, \emph{An Invitation to Modern Number Theory}, Princeton University Press, Princeton, NJ, 2006.


\bibitem[Mon]{Mon}
\newblock H. Montgomery, \emph{The pair correlation of zeros of the zeta
function}, Analytic Number Theory, Proc. Sympos. Pure Math.
\textbf{24}, Amer. Math. Soc., Providence, $1973$, $181-193$.

\bibitem[Od1]{Od1}
\newblock A. Odlyzko, \emph{On the distribution of spacings
between zeros of the zeta function}, Math. Comp. \textbf{48} (1987),
no. 177, 273--308.

\bibitem[Od2]{Od2}
\newblock A. Odlyzko, \emph{The $10^{22}$-nd zero of the Riemann zeta function}, Proc.
Conference on Dynamical, Spectral and Arithmetic Zeta-Functions, M.
van Frankenhuysen and M. L. Lapidus, eds., Amer. Math. Soc.,
Contemporary Math. series, 2001,
\texttt{http://www.research.att.com/$\sim$amo/doc/zeta.html}.

\bibitem[OS1]{OS1}
\newblock A. E. \"Ozl\"uk and C. Snyder, \emph{Small zeros of quadratic $L$-functions},
Bull. Austral. Math. Soc. \textbf{47} (1993), no. 2, 307--319.

\bibitem[OS2]{OS2}
\newblock A. E. \"Ozl\"uk and C. Snyder, \emph{On the distribution of the
non-trivial zeros of quadratic $L$-functions close to the real axis},
Acta Arith. \textbf{91} (1999), no. 3, 209--228.

\bibitem[RR]{RR}
\newblock G. Ricotta and E. Royer, \emph{Statistics for low-lying
zeros of symmetric power $L$-functions in the level aspect},
preprint. \texttt{http://arxiv.org/abs/math/0703760}

\bibitem[Ro]{Ro}
\newblock E. Royer, \emph{Petits z\'{e}ros de fonctions $L$
de formes modulaires}, Acta Arith. \textbf{99} (2001),  no. 2,
147-172.

\bibitem[Rub1]{Rub1}
\newblock M. Rubinstein, \emph{Low-lying zeros of $L$--functions
and random matrix theory}, Duke Math. J. \textbf{109}, (2001),
147--181.

\bibitem[Rub2]{Rub2}
\newblock M. Rubinstein, \emph{Computational methods and experiments
in analytic number theory}. Pages 407--483 in Recent Perspectives in
Random Matrix Theory and Number Theory, ed. F. Mezzadri and N. C.
Snaith editors, 2005.

\bibitem[RS]{RS}
\newblock Z. Rudnick and P. Sarnak, \emph{Zeros of principal $L$-functions
 and random matrix theory}, Duke Math. J. \textbf{81},
 $1996$, $269-322$.

\bibitem[Se]{Serre}
J. P. Serre, \emph{A Course in Arithmetic}, Springer-Verlag, 1973, 97.
\label{Serre}

\bibitem[Sh]{Sh}
G. Shimura, \emph{On the holomorphy of certain Dirichlet series}, Proc. Lond. Math. Soc. \textbf{31} (1975), no. 3, 79--98.

\bibitem[So]{So}
\newblock K. Soundararajan, \emph{Nonvanishing of quadratic
Dirichlet $L$-functions at $s=1/2$}, Ann. of Math. (2) \textbf{152}
(2000), 447--488.

\bibitem[TW]{TW}
R. Taylor and A. Wiles, \emph{Ring-theoretic properties of certain
Hecke algebras}, Ann. Math. \textbf{141} (1995), 553--572.

\bibitem[Wig1]{Wig1}
E. Wigner, \emph{On the statistical distribution of the widths and
spacings of nuclear resonance levels}, Proc. Cambridge Philo. Soc.
\textbf{47} (1951), 790--798.

\bibitem[Wig2]{Wig2}
E. Wigner,\ \emph{Characteristic vectors of bordered matrices with
infinite dimensions}, Ann. of Math. \textbf{2} (1955), no. 62,
548--564.

\bibitem[Wig3]{Wig3}
E. Wigner, \emph{Statistical Properties of real symmetric
matrices}. Pages 174--184 in \emph{Canadian Mathematical Congress
Proceedings}, University of Toronto Press, Toronto, 1957.

\bibitem[Wig4]{Wig4}
E. Wigner, \emph{Characteristic vectors of bordered matrices with
infinite dimensions. II}, Ann. of Math. Ser. 2 \textbf{65} (1957),
203--207.

\bibitem[Wig5]{Wig5}
E. Wigner, \emph{On the distribution of the roots of certain
symmetric matrices}, Ann. of Math. Ser. 2 \textbf{67} (1958),
325--327.


\bibitem[Wi]{Wi}
A. Wiles, \emph{Modular elliptic curves and Fermat's last
theorem}, Ann. Math. \textbf{141} (1995), 443--551.

\bibitem[Wis]{Wis}
J. Wishart, \emph{The generalized product moment distribution in
samples from a normal multivariate population}, Biometrika
\textbf{20 A} (1928), 32--52.


\bibitem[Yo1]{Yo1}
\newblock M. Young, \emph{Lower-order terms of the 1-level density of families of elliptic
curves},  Internat. Math. Res. Notices 2005,  no. 10, 587--633.

\bibitem[Yo2]{Yo2}
\newblock M. Young, \emph{Low-lying zeros of families of elliptic curves},
J. Amer. Math. Soc. \textbf{19} (2006), no. 1, 205--250.


\end{thebibliography}
\end{document}